\documentclass{fu-agplus-proceedings-arxiv}
\usepackage[utf8]{inputenc}

\usepackage{mathtools}

\theoremstyle{theorem}
\newtheorem{alsproposition}[thm]{Proposition}
\theoremstyle{definition}

\numberwithin{thm}{section}

\begin{document}
\title{D-Modules and Holonomic Functions}





\proctitle        {D-Modules and Holonomic Functions}
\proclecturer     {Anna-Laura Sattelberger and Bernd Sturmfels}
\procinstitution  {MPI-MiS Leipzig and MPI-MiS Leipzig/UC Berkeley}
\procmail         {anna-laura.sattelberger@mis.mpg.de and bernd@mis.mpg.de}
\procauthor       {Anna-Laura Sattelberger and Bernd Sturmfels}
\procabstract     {%
	In algebraic geometry, one studies  the solutions to polynomial equations, or, equivalently,
	to  linear partial differential equations with 
	constant coefficients. These lecture notes address the more general case 
	when the coefficients of the partial differential equations are polynomials. The letter $D$ 
	stands for the Weyl algebra, and a $D$-module is a left module over~$D$. We 
	focus on left ideals, or $D$-ideals. 
	We represent holonomic functions in several variables by 
	the linear differential equations they satisfy. This
	encoding by a $D$-ideal is useful~for many problems, e.g.,~in
	geometry, physics and statistics. We explain how to~work~with holonomic functions. 
	Applications include volume computations and likelihood inference.
}

\makeheader

\section*{Introduction}
This article represents the notes for the three lectures delivered by the
second author at the Math+ Fall School in Algebraic Geometry,
held at FU Berlin from September~30 to October~4, 2019.
The aim was to give an introduction to \mbox{$D$-modules} and their use for
concrete  problems in applied algebraic geometry. This centers around
the concept of a holonomic function in several variables. Such a function
is among the solutions to a system of linear partial differential equations
whose solution space is finite-dimensional.
Our algebraic representation as a \mbox{$D$-ideal} allows us to perform
many operations with such objects.

The three lectures form the three sections of this article.
In the first lecture, we introduce the basic tools. The focus is on
computations that are built on Gr\"obner bases in the Weyl algebra.
We review the Fundamental Theorem of Algebraic Analysis,
we introduce holonomic $D$-ideals, and we discuss their holonomic rank.
Our presentation of these topics is based on the book \cite{als:SST}.

In the second lecture, we study the problem of measuring areas and volumes
of semi-algebraic sets. These quantities are  represented as 
integrals, and our task is to evaluate such an integral as accurately as is possible.
To do so, we follow the approach of Lairez, Mezzarobba, and Safey El Din
\cite{als:LMS}. We introduce a parameter into the integrand,
and we regard our integral as a function in that parameter.
Such a volume function is holonomic, and we derive a $D$-ideal that annihilates it.
Using manipulations with that $D$-ideal, we arrive at a scheme that
allows for highly accurate numerical evaluation of the relevant integrals.

The third lecture is about connections to statistics.
Many special functions arising in statistical inference are
holonomic. We start out with likelihood functions for discrete models
and their Bernstein--Sato polynomials, and we then discuss the
{\em holonomic gradient method} and
{\em holonomic gradient descent}.
These were developed by a group of Japanese scholars \cite{als:HNTT, als:HGM}
for representing, evaluating and optimizing functions arising in statistics.
We give an introduction to this theory,
aiming to highlight opportunities for further research.
Our readers can hone their holonomic skills with a list of $19$ problems at the end of these lecture notes. 
We also provide hints and solutions to solve the problems.

\section{Tools}

Our presentation follows closely the one in \cite{als:SST, als:Tak}, albeit we use a slightly 
different notation.
For any integer $n \geq 1$, we introduce the $n$th {\em Weyl algebra} 
with complex coefficients:
\[ D \,\, = \,\, \mathbb{C}  \left[ x_1,\ldots,x_n\right] \langle  \partial_1, \ldots, \partial_n \rangle . \] 
We  sometimes write $D_n$ instead of $D$ if we want to highlight
the dimension $n$ of the ambient space.
For the sake of brevity, we often write $D = \mathbb{C} [x]\langle \partial \rangle$,
especially when $n=1$.
Formally, $D$ is the free associative algebra over $\mathbb{C} $ in the $2n$ generators 
$x_1,\ldots,x_n,\partial_1,\ldots,\partial_n$ modulo the relations
that all pairs of generators commute, with the exception of the $n$ 
special relations
\begin{equation}\label{als:eqprodcalc} 
\quad \partial_i x_i - x_i \partial_i \,\, = \,\, 1 \qquad \hbox{for}\,\,\, i =1,2,\ldots,n. 
\end{equation}
The Weyl algebra $D$ is similar to a commutative polynomial ring in $2n$ variables.
But it is non-commutative,
due to~\eqref{als:eqprodcalc}. A $\mathbb{C} $-vector space basis of $D$ consists of the 
{\em normal monomials}
\[ x^a \partial^b \,\,=\,\, x_1^{a_1} x_2^{a_2} \cdots x_n^{a_n}
\partial_1^{b_1} \partial_2^{b_2} \cdots \partial_n^{b_n}. \]
Indeed, every word in the $2n$ generators can be written uniquely as an
integer linear
combination of normal monomials. It is an instructive
exercise to find this expansion for a monomial $ \partial^u x^v$.
Many computer algebra systems have a built-in capability for computing in $D$.
For instance, in {\tt Macaulay2} (cf.~\cite{als:LT}),
 the expansion into normal monomials is 
done automatically.

\begin{ex} Let $n=2$, $u=(2,3)$ and $v = (4,1)$.
	The normal expansion of $\partial^u x^v= \partial_1^2  \partial_2^3 x_1^4 x_2$ 
	can be found by hand:
	\begin{equation}
	\label{als:eqexpansion} 
	x_1^4 x_2  \partial_1^2  \partial_2^3
	+3 x_1^4  \partial_1^2  \partial_2^2
	+8 x_1^3 x_2  \partial_1  \partial_2^3
	+24 x_1^3  \partial_1  \partial_2^2
	+12 x_1^2 x_2  \partial_2^3
	+36 x_1^2  \partial_2^2.
	\end{equation}
		For this derivation
	we recommend the following intermediate factorization:
	\[ ( \partial_1^2 x_1^4) (\partial_2^3 x_2) \,\, = \,\,
	\bigl( x_1^4 \partial_1^2 + 8 x_1^3 \partial_1 + 12 x_1^2\bigr)
	\bigl( x_2 \partial_2^3 + 3 \partial_2  \bigr). \]
	The formula~\eqref{als:eqexpansion} is the output when the following line 
	is typed into {\tt Macaulay2}:
	\begin{verbatim}
	D = QQ[x1,x2,d1,d2, WeylAlgebra => {x1=>d1,x2=>d2}];  
	d1^2*d2^3*x1^4*x2
	\end{verbatim}
\end{ex}
Another important object is the ring of linear differential operators 
whose coefficients are rational functions
in $n$ variables. We call this ring the  {\em rational Weyl algebra} and we 
denote it~by 
\[ \,\,\, R \,\,= \,\,\mathbb{C}  (x_1,\ldots,x_n)\langle \partial_1, \ldots, \partial_n \rangle
\qquad \text{or} \qquad R \,=\, \mathbb{C} (x)\langle \partial \rangle. \]
Note that $D$ is a subalgebra of $R$.
The multiplication in $R$ is defined as follows:
\begin{equation}
\label{als:eqprodcalc2}
\partial_i r(x) \,\, = \,\, r(x) \partial_i + \frac{\partial r}{\partial x_i}(x)
\qquad \hbox{for all }\,\, r \in \mathbb{C} (x_1,\ldots,x_n). 
\end{equation}
This simply extends the product rule~\eqref{als:eqprodcalc} from polynomials
to rational functions. The ring $R_1$ is a (non-commutative) principal ideal domain, 
whereas $D_1$ and $R_2$ are not. But, a theorem due to Stafford~\cite{als:Sta} guarantees
that every $D$-ideal is generated by only two elements.

We are interested in studying left modules $M$ over the Weyl algebra $D$ 
or the rational Weyl algebra $R$. 
Throughout these lecture notes, we denote the action of $D$ resp.~$R$ on $M$~by 
\[ \bullet \colon D\times M \to M \quad \text{(resp. } \bullet \colon R\times M \to M \text{)}.\]
We are especially interested in
left  $D$-modules of the form $D/I$ for some 
left ideal $I$ in the Weyl algebra $D$. 
Systems of linear partial differential equations with polynomial coefficients then 
can be investigated as modules over $D$. 
Likewise, rational coefficients lead to modules over $R$.
Therefore, the theory of $D$-modules allows us to study linear PDEs with 
polynomial coefficients by algebraic methods. 
In these notes we will exclusively deal with {\em left} modules over $D$ 
(resp. {\em left} ideals in $D$) and will refer to them simply as
{\em $D$-modules} (resp. {\em $D$-ideals}).

\begin{remark}
	In many sources, the theory of $D$-modules is introduced  more abstractly.
	Namely, one considers the sheaf $\mathcal{D}_X$ of differential operators
	on some smooth complex variety $X$. Its sections on an affine open subset $U$ is
	the ring $\mathcal{D}_X(U)$ of differential operators on the corresponding 
	$\mathbb{C} $-algebra $\mathcal{O}_X(U)$.
	In our case,  $X = \mathbb{A}_{\mathbb{C} }^n$ is affine $n$-space 
	over the complex numbers,
	so $\mathcal{O}_X(X) = \mathbb{C} [x_1,\ldots,x_n]$ is the polynomial ring,
	and the Weyl algebra is recovered as the global sections of that sheaf.
	In symbols, we have $D = \mathcal{D}_X(X)$.	Modules over~$D$ then 
	correspond precisely to $\mathcal{O}_X$-quasi-coherent 
	$\mathcal{D}_X$-modules, see~\cite[Proposition 1.4.4]{als:HTT}.
\end{remark}

\begin{ex} 
	Many function spaces are $D$-modules in a natural way. 
	Let $F$ be a space of holomorphic functions 
	on a domain in $\mathbb{C} ^n$, such that $F$ is closed under taking partial derivatives.
	The natural action of the Weyl algebra turns~$F$ into a $D$-module, 
	as follows:
	\[ \bullet \colon D \times F \longrightarrow F, \quad 
	(\partial_i ,f) \,\mapsto \,
	\partial_i \bullet f \,\coloneqq \,\frac{\partial f}{\partial x_i}, 
	\  (x_i,f) \mapsto x_i \bullet f \,\coloneqq \,x_i\cdot f.\]
\end{ex}

\begin{defn}
	We write $\text{Mod}(D)$ for the category of $D$-modules.
	Let $I$ be a $D$-ideal and $M\in \text{Mod}(D)$. The {\em solution space of $I$ in $M$} 
	is the $\mathbb{C}  $-vector space
	\[ \text{Sol}_M(I)
	\,\,\coloneqq \,\,\left\{ m \in M \mid P \bullet m =0 \text{ for all } P\in I\right\} .\]
\end{defn}

\begin{remark} 
	For $P \in D, \,M\in \text{Mod}(D)$, we have the vector space~isomorphism
	\[ \text{Hom}_D\left( D/DP,M\right)\,\, \cong \,\,\left\{ m \in M \mid P \bullet m=0 \right\}. \]
	This implies that ${\rm Sol}_M(I)$ is isomorphic to 
	$ \text{Hom}_D\left( D/I,M \right)$.
\end{remark}

In what follows, we will be relaxed about specifying the
function space $F$ or module $M$. The theory works best for holomorphic 
functions on a small open ball in $\mathbb{C} ^n$.
For our applications in the later sections, we think of smooth
real-valued functions on an open subset of $\mathbb{R} ^n$,
and we usually  take $D$ with coefficients in~$\mathbb{Q} $.
We will often drop the subscript $F$ in ${\rm Sol}_F(I)$
and assume that a suitable class $F$ of 
infinitely differentiable functions is understood from the context.

\begin{ex}[$n=2$] 
	Let $I=\langle \partial_1x_1\partial_1,\partial_2^2+1\rangle$. 
	The solution space equals
	\[ \text{Sol}(I) \,\,=\,\, \mathbb{C}  \bigl\{ \,\sin(x_2),\,\cos (x_2),\, \log (x_1)\sin (x_2),\,
	\log (x_1)\cos(x_2) \bigr\}.\]
	Hence ${\rm dim}( \text{Sol}(I)) = 4$.
	The reader is urged to verify this and to experiment with questions like these: What happens if 
	$\partial_2^2+1$ is replaced by
	$\partial_2^3+1$? What if
	$\,\partial_1x_1\partial_1$ is replaced by 
	$\,\partial_1 x_1 \partial_1x_1\partial_1$, or if some indices $1$ are turned into $2$?
\end{ex}

The Weyl algebra $D$ has three important
commutative polynomial subrings:
\begin{bulist}
	\item The usual polynomial ring $\mathbb{C}  [x_1,\ldots,x_n]$
	acts by multiplication on function spaces. Its ideals represent subvarieties of $\mathbb{C} ^n$.
	In analysis, this models
	distributions that are supported on subvarieties.
	\item The polynomial ring $\mathbb{C}  [\partial_1,\ldots,\partial_n]$ 
	represents linear partial 
	differential operators with constant coefficients. 
	Solving such PDEs is very interesting already.
	This is highlighted in
	\cite[Chapter~10]{als:Stu}.
	\item The polynomial ring $\mathbb{C}  [\theta_1,\ldots,\theta_n]$, 
	where $\theta_i \coloneqq x_i \partial_i$, will be important for us shortly.
	The sum $\theta_1 + \cdots + \theta_n$ is known as the	{\em Euler operator}.	
	The joint eigenvectors of the operators
	$\theta_1,\ldots,\theta_n$  are precisely the monomials.
\end{bulist}

\begin{ex}[$n=1$]\label{als:exreplacing} 
	Replacing the derivatives $\partial_i$ by the operators $\theta_i$ has an interesting
	effect on the solutions. It replaces each coordinate $x_i$ by ${\rm log}(x_i)$.
	In one variable, we suppress the indices:
	\begin{bulist}
		\item For $I = \langle (\partial +3)^2(\partial-7) \rangle $,
		$\text{Sol} (I) = \mathbb{C}  \left\{ \exp(-3x),x\exp(-3x),\exp(7x) \right\}$.
		\item For $J = \langle (\theta +3)^2(\theta - 7) \rangle$,
	    $\text{Sol}(J) = \mathbb{C}  \left\{ x^{-3},\log(x)x^{-3},x^7  \right\} $.
	\end{bulist}
	The reader should try the same for their favorite polynomial ideal in $n$ variables.
\end{ex}

We already mentioned that
every operator $P$ in the Weyl algebra $D$ 
has a unique expansion 
into normally ordered monomials,
\[ P\,\,=\sum_{(a,b)\in E }c_{ab}x^a\partial^b,\] where $c_{ab}\in \mathbb{C}  \backslash \{0\}$ 
and $E$ is a finite subset of $\mathbb{N} ^{2n}$.
Fix $u,v\in \mathbb{R} ^n$ with $u+v \geq 0$ and set
$m=\max_{(a,b)\in E}\left( a\cdot u + b\cdot v\right)$.
The {\em initial form} of $\,P \in D\,$ is defined as
\[ \text{in} _{(u,v)}(P) \quad = \sum_{a\cdot u + b\cdot v =m} \!\! c_{ab}\prod_{u_k+v_k>0}
\!\! x_k^{a_k}\xi_k^{b_k} \prod_{u_k+v_k=0} \!\! x_k^{a_k}\partial_k^{b_k}
\,\,\,\in \,\,\, \text{gr}_{(u,v)}\left( D \right).\]
Here $\xi_k$ is a new variable
that commutes with all others. The initial form~is an element in the
{\em associated graded ring} under
the filtration of $D$ by the weights~$(u,v)$:
\[ \text{gr}_{(u,v)}(D) \,\,=\,\,
\begin{cases}
D & \text{ if } u+v=0,\\
\mathbb{C} [x_1,\ldots,x_n,\xi_1,\ldots,\xi_n] & \text{ if }u+v>0, \\
\text{a mixture of the above} & \text{ otherwise.} 
\end{cases}\]
Of particular interest is the case when
$u$ is the zero vector and $v$ is the all-one
vector $e = (1,1,\ldots,1)$.
Analysts refer to ${\rm in}_{(0,e)}(P)$ as the {\em symbol}
of the differential operator $P$. Thus, the symbol of
a differential operator in $D$ is simply an ordinary polynomial in $2n$ variables.

We continue to allow arbitrary weights $u,v\in \mathbb{R} ^n$ that satisfy $u+v \geq 0$.
For a $D$-ideal~$I$,
the vector space $ \text{in} _{(u,v)}(I) = \mathbb{C}  \left\{ \text{in} _{(u,v)}(P) \mid P\in I \right\} $ 
is a left ideal in $\text{gr}_{(u,v)}(D)$.
The computation of the {\em initial ideals}
$ \text{in} _{(u,v)}(I)$ and their associated {\em Gr\"obner bases} via
the {\em Buchberger algorithm} in $D$  is the engine behind many
practical applications of $D$-modules, such as those in
\cite{als:ALSS,als:HNTT,als:HGM}.

\begin{defn} 
	Fix $e = (1,1,\ldots,1)$ and $0 =(0,0,\ldots,0)$ in $\mathbb{R} ^n$.
	Given any $D$-ideal $I$,
	the {\em characteristic variety} $\text{Char}(I)$ is the vanishing set of
	the {\em characteristic ideal} $\,\text{ch} (I) \coloneqq \text{in} _{(0,e)}(I)$ in $\mathbb{C} ^{2n}$.
	The characteristical ideal is the ideal in\linebreak$\mathbb{C} [x,\xi] = 	\mathbb{C} [x_1,\ldots,x_n,\xi_1,\ldots,\xi_n] $
	which is generated by the symbols ${\rm in}_{(0,e)}(P)$ of all differential operators 
	$P\in I$. 
\end{defn}
In the theory of $D$-modules, one refers to
$\text{Char}(I)$ as the characteristic variety of the $D$-module $D/I$.

\begin{ex}[$n=2$]\label{als:extwoplanes} 
	Let $I=\langle x_1\partial_2,x_2\partial_1 \rangle$.  The two generators
	are \underbar{not} a Gr\"obner basis of~$I$. To get a Gr\"obner basis, one also needs
	their commutator. The characteristic ideal~is
	\begin{align*} \text{ch} (I) &\,\,=\,\, \langle x_1\xi_2,x_2\xi_1,x_1\xi_1-x_2\xi_2 \rangle \\
	&\,\,=\,\, \langle x_1,x_2 \rangle 
	\,\cap\, \langle \xi_1,\xi_2 \rangle \,\cap\, \langle 
	x_1^2  ,x_2^2,x_1\xi_2,x_2\xi_1,\xi_1^2, \xi_2^2, x_1\xi_1-x_2\xi_2 \rangle .
	\end{align*}
	The last ideal is primary to the embedded prime $\langle x_1,x_2,\xi_1,\xi_2 \rangle$.
	The characteristic variety is the union of two planes,
	defined by the two minimal primes, that meet in the origin in $\mathbb{C} ^4$. 
\end{ex}

\begin{thm}[Fundamental Theorem of Algebraic Analysis]
	${}$ \newline
	Let $I$ be a proper $D$-ideal. Every irreducible component of  
	its characteristic variety $\text{Char} (I)$ has dimension at least $n$.  
\end{thm}
This theorem was established by Sato, Kawai, and Kashiwara in~\cite{als:SKK}. 
\begin{defn} 
	A $D$-ideal $I$ is {\em holonomic} if $\dim \left(\text{ch}(I)\right)=n$, i.e.,~if
	the dimension of its characteristic variety in $\mathbb{C} ^{2n}$ is minimal.
	Fix the field $\,\mathbb{C} (x) = \mathbb{C} (x_1,\ldots,x_n)\,$ and 
	$\,\mathbb{C} (x)[\xi]=\mathbb{C} (x)[\xi_1,\ldots,\xi_n]$. 
	The {\em holonomic rank} of $I$~is 
 	\[ \, \text{rank}  (I) \, = \, \dim_{\mathbb{C}  (x)} \bigl( \mathbb{C} (x)[\xi]/ \mathbb{C} (x)[\xi] \text{ch}(I) \bigr) 
	\, = \, \dim_{\mathbb{C} (x)} \left( R/RI\right).\]
	Both dimensions count the
	standard monomials for a Gr\"obner basis
	of $RI$ in $R$ with respect to $e$.
	The $D$-ideal in Example~\ref{als:extwoplanes} is holonomic
	of holonomic rank~$1$.
\end{defn}

If a $D$-ideal $I$ is holonomic, then $\text{rank} (I)<\infty$. 
But, the converse is not true. For instance, consider
the $D$-ideal  $I=\langle x_1\partial_1^2,x_1\partial_2^3\rangle$.
Its characteristic ideal equals
${\rm ch}(I) = \langle x_1 \rangle \cap \langle \xi_1^2, \xi_2^3 \rangle$,
so $ \text{Char}(I)$ has dimension $3$ in $\mathbb{C} ^4$, which means that
$I$ is not holonomic. After passing to rational function coefficients,
we have $\mathbb{C} (x)[\xi]  {\rm ch}(I) = \langle \xi_1^2, \xi_2^3 \rangle$, 
and hence ${\rm rank}(I) = 6 < \infty$.

We define the \emph{Weyl closure} of  a $D$-ideal $I$ 
to be the $D$-ideal $W(I) \coloneqq RI\cap D$.
We always have $I \subseteq W(I)$, and 
$I$ is \emph{Weyl-closed} if $I = W(I)$ holds.
The operation of passing to the Weyl closure 
is analogous to that of passing to the radical in a polynomial ring.
Namely, assuming $\text{rank}(I)$ is finite,
$W(I)$ is the ideal of all differential operators that annihilate
all classical solutions of $I$. In particular, it fulfills
${\rm rank}(I) = {\rm rank}(W(I))$. If
$I=\langle x_1\partial_1^2,x_1\partial_2^3\rangle$,
then $W(I) = \langle \partial_1^2, \partial_2^3 \rangle$, and
\[ {\rm Sol}(I) \, = \, {\rm Sol}(W(I)) \, = \,
\mathbb{C}  \{ 1, x_1, x_2 , x_1 x_2, x_2^2 , x_1 x_2^2 \}. \]
In general, it is a difficult task to compute the
Weyl closure $W(I)$ from given generators of~$I$.

\begin{defn} 
	The \emph{singular locus} $\text{Sing} (I)$ of $I$ is the variety in $\mathbb{C} ^{n}$ 
	defined~by 
	\begin{equation}\label{als:eqsing}
	\text{Sing}  (I) \,\,\coloneqq \,\,\bigl(  \text{ch}(I): \langle \xi_1,\ldots,\xi_n
	\rangle^{\infty}  \bigr) \,\cap \,\mathbb{C} [x_1,\ldots,x_n].
	\end{equation}
	Geometrically, the singular locus is the closure of the projection
	of \linebreak$\text{Char}(I) \backslash (\mathbb{C} ^n \times \{0\})$ onto the
	first $n$ coordinates of $\mathbb{C} ^{2n}$.
	If $I$ is holonomic, then $\text{Sing} (I)$ is a proper subvariety of $\mathbb{C} ^n$.
	For instance, in Example~\ref{als:exreplacing} with $n=1$,
	we have ${\rm Sing}(I) = \emptyset$ whereas 
	${\rm Sing}(J) = \{ 0 \}$ in $\mathbb{C} ^1$.
\end{defn}

\begin{ex}[$n=4$]\label{als:eqGauss}
	Fix constants $a,b,c\in \mathbb{C} $ and consider the $D$-ideal
	\begin{equation}
	\label{als:eqgauss4}
	I\,\,=\,\,\langle \,
	\, \theta_1-\theta_4+1-c,\,\theta_2+\theta_4+a,\,\theta_3+\theta_4+b\, ,
	\, \partial_2\partial_3-\partial_1\partial_4\rangle.
	\end{equation}
	The first three operators tell us that every solution $g \in {\rm Sol}(I)$
	is $\mathbb{Z} ^3$-homogeneous:
	\[ g(x)\,=\,x_1^{c-1}x_2^{-a}x_3^{-b}  \cdot f \left(\frac{x_1x_4}{x_2x_3} \right) 
	\qquad \hbox{for some function $f$.} \]
	The last generator of the $D$-ideal $I$ 
	implies that the univariate function $f$ satisfies
	\[ x(1-x)f'' +(c-x(a+b+1))f'-abf=0.\]
	This second-order ODE is {\em Gau{\ss}' hypergeometric equation}.
	The $D$-ideal $I$ is the 
	Gel'fand--Kapranov--Zelevinsky
	(GKZ) representation of Gau{\ss}' hypergeometric function. We have 
	\[ \text{ch} (I) \,\,= \,\,\langle \xi_2\xi_3-\xi_1\xi_4,\,x_1\xi_1-x_4\xi_4,x_2\xi_2+x_4\xi_4,x_3\xi_3+x_4\xi_4\rangle. \]
	This characteristic ideal has ten associated primes, one for each face of a square. 
	Using the ideal operations in~\eqref{als:eqsing} we find 
	$\text{Sing} (I) = \langle \,x_1x_2x_3x_4(x_1x_4-x_2x_3)\,\rangle$.
	The generator is the  {\em principal determinant} of the
	matrix $\begin{pmatrix} x_1 & x_2 \\ x_3 & x_4 \end{pmatrix}$,
	i.e.,~the product of all of its subdeterminants.
\end{ex}

\begin{thm}[Cauchy--Kowalevskii--Kashiwara]
	Let $I$ be a holonomic \mbox{$D$-ideal} and let $\,U$ be an open subset 
	of $ \, \mathbb{C} ^n \setminus \text{Sing} (I)$ that is simply connected. Then
	the space of holomorphic functions on $U$ that are solutions to $I$ 
	has dimension equal to $\text{rank} (I)$.
	In symbols, ${\rm dim}({\rm Sol}(I)) = \text{rank} (I)$.
\end{thm}

For a discussion and pointers to proofs, we refer to \cite[Theorem 1.4.19]{als:SST}.
Let us point out that the theorem also holds true if only $\text{rank}(I)<\infty$.

\begin{ex} 
	The $D$-ideal $I$ in~\eqref{als:eqgauss4} is holonomic
	with ${\rm rank}(I)=2$. Outside the
	singular locus in $\mathbb{C} ^4$, it has two 
	linearly independent solutions.
	These arise from Gau\ss' \ hypergeometric ODE.
\end{ex}

This raises the question of how one can compute a basis
for ${\rm Sol}(I)$ when $I$ is holonomic.
To answer this question,  let us think about the case $n=1$.
Here, the classical {\em Frobenius method} can be used to
construct a basis consisting of series solutions. These have the form
\begin{equation}\label{als:eqfrobf}
p(x) \quad = \quad x^a {\rm log}(x)^b \,+\, \hbox{higher order terms} .
\end{equation}
The first exponent $a$ is a complex number. It is a zero 
of the {\em indicial polynomial} of the given ODE.
The second exponent $b$ is a nonnegative integer, strictly
less than the multiplicity of $a$ as a zero of the indicial polynomial.

\begin{ex}\label{als:exfrobn=1} 
This following ordinary differential equation
 appears in the 2019 {\tt Wikipedia} entry for 
``Frobenius method''.
	We are looking for a solution $p$ to 
	$$ x^2 p''-x p'+(1-x) p\,\,=\,\,0 .$$ This equation is equivalent to 
	$p$ being a solution of $\,I = \langle \theta^2-2\theta+1-x \rangle $. 
	The indicial polynomial is the generator of the
	initial ideal  ${\rm in}_{ (-1,1)} (I) = \langle \,(\theta - 1)^2 \, \rangle $.
	The basis of solutions \eqref{als:eqfrobf}
	 consists of two series
	with $a = 1$ and $b \in \{0,1\}$. The higher order terms
	of these series are computed by solving the linear recurrences
	for the coefficients that are induced by $I$.
\end{ex}

Fix an integer $n \geq 1$.
The algebraic $n$-torus $T\coloneqq (\mathbb{C}  \backslash \{0\})^n$ 
acts  naturally on the Weyl algebra $D$, by scaling the generators
$\partial_i$ and $x_i$ in a reciprocal manner:
\[\circ \,\colon T \times D \longrightarrow D, \quad (t,\partial_i) 
\mapsto t_i\partial_i, \ (t,x_i)\mapsto \frac{1}{t_i} x_i .\]
This action is well-defined because it preserves the defining 
relations~\eqref{als:eqprodcalc}.
A $D$-ideal $I$ is said to be \emph{torus-fixed} if $t\circ I = I$ for all $t\in T$.  
Torus-fixed $D$-ideals play the role of monomial ideals in an 
ordinary commutative polynomial ring.
They can be described as follows:

\begin{alsproposition}\label{als:prop3cond}
	Let $I$ be a $D$-ideal. The following conditions are equivalent:
	\vspace*{-0.2in}
	\begin{abclist}
		\item $I$ is torus-fixed, \vspace{-0.1in}
		\item $I=\text{in} _{(-w,w)}(I)$ for all $w\in \mathbb{R} ^n$, \vspace{-0.1in}
		\item $I$ is generated by operators $x^ap(\theta)\partial^b$
		where $a,b\in \mathbb{N}  ^n$ and $p\in \mathbb{C} [\theta]$.
	\end{abclist}
\end{alsproposition}

Just like in the commutative case, initial ideals
with respect to generic weights are torus-fixed.
Namely, given any $D$-ideal $I$,
if $w$ is generic in $\mathbb{R} ^n$, then $\text{in} _{(-w,w)}(I)$ is torus-fixed.
Note that $\text{in} _{(-w,w)}(I)$  is also a $D$-ideal, and 
it is important to understand how its solution space
is related to that of $I$.
Lifting the former to the latter is the key to the Frobenius method.

Let $w \in \mathbb{R} ^n$, and suppose that $f$ is a
holomorphic function on an appropriate open subset of~$\mathbb{C} ^n$.
We assume that 
a series expansion of $f$ has the lowest order 
form $\text{in} _w(f)$, called {\em initial series} with respect to assigning the weight $w_i$ for
the coordinate $x_i$ for all $i$ (see~\cite[Definition 2.5.4]{als:SST} for details). 
Under this assumption, the following result holds:

\begin{alsproposition}\label{als:propdistrac}
	If $\,f\in \text{Sol}(I)$, then $\,\text{in} _w(f)\in \text{Sol} \bigl( \text{in} _{(-w,w)}(I)\bigr)$.
\end{alsproposition}

This fact is reminiscent of the use of {\em tropical geometry} \cite{als:MS} in solving
polynomial equations. The tropical limit of an ideal in 
a (Laurent) polynomial ring is given by a monomial-free
initial ideal. The zeros of the latter furnish the starting terms in a series solution
of the former, where each series is a scalar in a field like
the Puiseux series or the $p$-adic numbers.

\smallskip 

We now focus on the polynomial subring
$\mathbb{C} [\theta] = \mathbb{C} [\theta_1,\ldots,\theta_n]$ of~$D$.

\begin{defn} 
	The \emph{distraction} of a $D$-ideal $I$ is the polynomial ideal 
	\[ \,\widetilde{I} \,\,\coloneqq\,\, RI \cap \mathbb{C} [\theta].\]
\end{defn}

By definition, $\widetilde{I}$ is contained in the Weyl closure $W(I)$.
If $I$ is torus-fixed, then $I$ and~$\widetilde{I}$ have the same Weyl closure,
$W(I) = W(\widetilde{I})$, by Proposition~\ref{als:prop3cond}.
This  means that all the classical solutions to the torus-fixed $D$-ideal $I$
are  represented by an ideal in the commutative polynomial ring $\mathbb{C} [\theta]$.

\begin{thm}\label{als:thmdistrac}
	Let $I$ be a torus-fixed ideal, with generators $\,x^ap(\theta)\partial^b\,$
	for various $a,b\in \mathbb{N}  ^{n}$. The distraction $\widetilde{I}$  is generated by 
	the corresponding polynomials
	$\,[\theta]_b \cdot p(\theta-b)$. Here we use the following notation for falling factorials:
	$$ [\theta]_b \,\,=\,\, \prod_{i=1}^n \prod_{j=0}^{b_i-1}(\theta_i-j). $$
\end{thm}

It is instructive to study the case when $I$ is generated by an Artinian
monomial ideal in $\mathbb{C} [\partial]$. In this case, $\widetilde{I}$ is the radical ideal
whose zeros are the nonnegative lattice points under the staircase diagram of $I$.

\begin{ex}[$n=2$]\label{als:staircase}
	Let $I = \langle \partial_1^3,\partial_1\partial_2,\partial_2^2\rangle$. 
	The points under the staircase diagram of $I$ are
	$(0,0),(1,0),(2,0),(0,1)$. The distraction is the radical ideal
	\[ \begin{matrix} \widetilde{I} & = &
	\bigl\langle \, \theta_1 (\theta_1-1)(\theta_1-2)\,,\,
	\theta_1\theta_2\,,\,\theta_2(\theta_2-1)\,\bigr\rangle \qquad  \qquad \\
	& = &  
	\langle \theta_1, \theta_2 \rangle \,\cap \,
	\langle \theta_1-1, \theta_2 \rangle \,\cap \,
	\langle \theta_1-2, \theta_2 \rangle \,\cap \,
	\langle \theta_1, \theta_2 -1 \rangle .
	\end{matrix} \]
	The solutions are spanned by the standard monomials of $I$. We find that
	$$ \,\text{Sol}(I)\,\,=\,\,\text{Sol}(\widetilde{I})\,\,=\,\,\mathbb{C} \left\{1,x_1,x_1^2,x_2 \right\}. $$
\end{ex}

We define a {\em Frobenius ideal} to be a $D$-ideal that is generated by
elements of the polynomial subring $\mathbb{C} [\theta]$. Thus, every torus-fixed ideal $I$ gives rise to
a Frobenius ideal $D \widetilde{I}$ that has the same classical solution space.
The solution space can be described explicitly when the given ideal in 
$\mathbb{C} [\theta]$ is Artinian.

Let $J$ be an Artinian ideal in the polynomial ring $\mathbb{C} [\theta] = \mathbb{C} [\theta_1,\ldots,\theta_n]$.
Then its  variety $V(J)$ is a finite subset of $\mathbb{C} ^n$. The primary decomposition of $J$ equals
\[ J \,\,= \bigcap_{u\in V(J)} Q_u(\theta - u), \] where $Q_u$ is an ideal
that is primary to the maximal ideal $\langle \theta_1,\ldots,\theta_n \rangle$ 
and \mbox{$Q_u(\theta-u)$} is the ideal obtained  from $Q_u$ 
by replacing $\theta_i$ with $\theta_i-u_i$ for\linebreak$i=1,\ldots,n$. 
We call $Q_u$ the primary component of $J$ at $u$. Its
\emph{orthogonal complement} is the
following finite-dimensional vector space
\begin{align*}
Q_u^{\perp} \,\,\coloneqq \,\,\left\{ \right. g\in \mathbb{C} [x_1,\ldots,x_n]  \,
\mid &  \ f(\partial_1,\ldots,\partial_n )\bullet g(x_1,\ldots,x_n) = 0\smallskip\\
 & \text{ for all } f=f(\theta_1,\ldots,\theta_n)\in Q_u \left. \right\} .
\end{align*}
In commutative algebra, the vector space $Q_u^{\perp}$ is known as the
{\em inverse system to the ideal $J$ at the given point~$u$}. Note that $Q_u^{\perp}$
is a module over $\mathbb{C} [\partial] = \mathbb{C} [ \partial_1,\ldots,\partial_n ]$.
If this module is cyclic, then $J$ is {\em Gorenstein} at $u$.
The \mbox{$\mathbb{C} $-dimension} of $Q_u^{\perp}$ is the 
{\em multiplicity} of the point $u$ as a zero of the ideal $J$.

\begin{thm}\label{als:solutionFrob}
	Given an ideal $J$ in the polynomial ring $ \mathbb{C} [\theta]$, the corresponding Frobenius ideal
	$I = D J$ in the Weyl algebra $D$ is holonomic if and only if $J$ is Artinian.
	In this case,
	$\,\text{rank} (I) = {\rm dim}_\mathbb{C}  \bigl( \mathbb{C} [\theta]/J \bigr)$, and the solution space
	$\,{\rm Sol}(I)$ is spanned by the functions
	\[ x_1^{u_1} \cdots x_n^{u_n} \cdot g( {\rm log}(x_1),\ldots,{\rm log}(x_n)), \]
	where $u \in V(J)$ and $g\in Q_u^{\perp}$ runs over
	a basis of the
	inverse system to $J$ at $u$.
\end{thm}

The theorem implies that the space of purely logarithmic solutions $g$ is given by the primary
component at the origin. Hence, it is interesting to study ideals that are
primary to $\langle \theta_1,\ldots,\theta_n \rangle$. 

\begin{ex}[$n=3$] The ideal
	$\,J=\langle \theta_1+\theta_2+\theta_3,\theta_1\theta_2+
	\theta_1\theta_3+\theta_2\theta_3,\theta_1\theta_2\theta_3\rangle\,$ 
	is generated by non-constant symmetric polynomials. 
	It is primary to $\langle \theta_1, \theta_2,\theta_3 \rangle$.
	We have $\mathcal{V}(J)=\{(0,0,0)\}$, with $\text{rank} (J)=6$.
	The inverse system is the \mbox{$6$-dimensional} space spanned
	by all polynomials that are 
	successive partial derivatives of $(x_1-x_2)(x_1-x_3)(x_2-x_3)$.
	This implies
	\[ \text{Sol}(J)\,\,=\,\,\mathbb{C} [\theta] \bullet \left(\log\left(\frac{x_1}{x_2}\right) \cdot \log
	\left(\frac{x_1}{x_3}\right)\cdot \log\left(\frac{x_2}{x_3}\right) \right) \,\, \cong \,\, \mathbb{C} ^6. \]
	Gorenstein ideals like $J$ arise in  applications of
	$D$-modules to mirror symmetry. 
	Some original sources for this connection are referred to in
\cite[page 150]{als:SST}.

\end{ex}

The term ``Frobenius ideal'' is a reference to the Frobenius method. This is a classical 
method for solving linear ODEs. The next theorem extends this
to PDEs. The role of the indicial polynomial
is now played by the indicial ideal.

\begin{thm}\label{als:indicial}
	Let $I$ be any holonomic $D$-ideal and $w\in \mathbb{R} ^n$ generic. The \emph{indicial ideal} 
	\[ \text{ind} _w(I) \,\,\coloneqq \,\,\widetilde{\text{in} _{(-w,w)}(I) }\]
	is a holonomic Frobenius ideal. Its rank equals the rank of $\,\text{in} _{(-w,w)}(I)$.
	This is bounded above by $\text{rank} (I)$, with equality when
	the $D$-ideal $I$ is regular holonomic.
\end{thm}

We refer to \cite[Section 2.4]{als:SST} for the definition of {\em regular holonomic}
and \cite[Theorem~2.5.1]{als:SST} for this result.
The indicial ideal $\text{ind} _w(I)$ is computed from $I$ by means of
Gr\"obner bases in $D$. This computation identifies
the leading terms in a basis of regular series solutions for~$I$.

\begin{ex}[$n=1$]\label{als:exwithlog}
	We illustrate Theorem~\ref{als:indicial} for the
	ODE in Example~\ref{als:exfrobn=1}.
	Let $I = \langle x^2 \partial^2 - x \partial + 1 - x \rangle$
	and set $w = 1$. The indicial ideal 
	$\, {\rm ind}_w(I) \,$ is the principal ideal in~$\mathbb{C} [\theta]$ 
	generated by $\theta^2 - 2 \theta + 1$.
	This polynomial has the unique root $u=1$ with
	multiplicity $2$. Hence, we obtain a basis of
	series solutions to $I$ which 
	take the form $x +  \cdots $ and $x \cdot {\rm log}(x) + \cdots$.
\end{ex}

\section{Volumes}
In calculus, we learn about definite integrals in order
to determine the area under a graph.
Likewise, in multivariable calculus, we examine the volume enclosed by a surface.
We are here interested in areas and volumes of semi-algebraic sets.
When these sets depend on one or more parameters, their volumes are
holonomic functions of the parameters. We explain what this means
and how it can be used for highly accurate evaluation of volume functions.

Suppose that $M$ is a $D$-module. We say that $M$
is {\em torsion-free} if it is torsion-free as a module over the polynomial ring
$\mathbb{C} [x] = \mathbb{C}  [x_1,\ldots,x_n]$.
In our applications, $M$ is usually a space of infinitely
differentiable or holomorphic  functions on a simply connected open set
in $\mathbb{R} ^n$ or $\mathbb{C} ^n$. Such $D$-modules are always torsion-free.
For a function $f \in M$, its  {\em annihilator}  is the $D$-ideal
\[ \text{Ann}_D\left( f \right) \,\,\coloneqq \,\,\left\{\, P\in D \mid P \bullet f =0 \,\right\}. \]
In general, it is a non-trivial task to compute the annihilating ideal. But, in some cases, 
computer algebra systems can help us to compute holonomic annihilating ideals. 
For rational functions $r \in \mathbb{Q} (x)$ this can be done using the package 
{\tt Dmodules} \cite{als:LT} in {\tt Macaulay2} 
with a built-in command as follows:
\begin{verbatim}
needsPackage "Dmodules"; 
D = QQ[x1,x2,d1,d2, WeylAlgebra => {x1=>d1,x2=>d2}];
rnum = x1;  rden = x2;   I = RatAnn(rnum,rden)
\end{verbatim}
Users of {\tt Singular} can do this with the library {\tt dmodapp.lib} \cite{als:AL}:
\begin{verbatim}
LIB "dmodapp.lib";
ring s=0,(x1,x2),dp; setring s;
poly rnum=x1; poly rden=x2;
def an=annRat(rnum,rden); setring an;
LD; 
\end{verbatim}
When you run this code, 
do try your own choice of numerator {\tt rnum} and 
denominator {\tt rden}. These should be polynomials
in the unknowns {\tt x1} and {\tt x2}. In our example we learn that $r = x_1/x_2$ has the annihilator
\[ {\rm Ann}_D(r) \,\,=\,\, \langle \,\partial_1^2 , \, x_1 \partial_1 - 1, \,
(x_2 \partial_2 + 1) \partial_1 \,\rangle . \]

Suppose now that $f(x_1,\ldots,x_n)$ is an algebraic function.
This means that $f$ satisfies some polynomial 
equation $F(f,x_1,\ldots,x_n)=0$. Using the polynomial $F$ as its input,
the {\tt Mathematica} package {\tt HolonomicFunctions} \cite{als:CK} can compute
a holonomic representation of $f$. In the univariate case, the output
is a linear differential operator of lowest degree annihilating $f$,
see Example~\ref{als:exmakeunique}.

Let $M$ be any $D$-module and 
$f\in M$. We say that $f$ is {\em holonomic} 
if ${\rm Ann}_D(f)$ is a holonomic $D$-ideal. If $f$ is an infinitely differentiable
function on an open subset of $\mathbb{R} ^n$ or $\mathbb{C} ^n$, then we refer to $f$ as a
{\em holonomic function}.

\begin{alsproposition}[\cite{als:GLS}]\label{als:holon}
	Let $f$ be an element in a torsion-free $D$-module $M$.
	Then the following three conditions are equivalent: \vspace{-0.05in}
	\begin{abclist}
		\item $f$ is holonomic, \vspace{-0.1in}
		\item $\text{rank}\left( \text{Ann}_{D}(f ) \right) < \infty$,\vspace{-0.1in}
		\item for each $i \in \{1,\ldots,n\}$ there exists an
		operator $P_i\in \mathbb{C}  [x_1,\ldots,x_n]\langle \partial_i \rangle 
		\backslash \{ 0 \}$ that annihilates~$f$. 
	\end{abclist}
\end{alsproposition}
\begin{proof}
	Let $I= \text{Ann} _D(f)$. If~$I$ is holonomic, then
	$RI$ is a zero-dimensional ideal in~$R$, i.e., $\dim_{\mathbb{C}(x)} (R/RI)<\infty.$  This condition is
	equivalent to b) and~c). 
	For the implication from b) to a), we note that $\text{Ann} _D(f)$ is Weyl-closed, 
	since $M$ is torsion-free. Finally,
	$\text{rank}\left( \text{Ann}_{D}(f ) \right) < \infty$ implies that
	$\text{Ann} _D(f)=W(\text{Ann} _D(f))$ is holonomic by \cite[Theorem~1.4.15]{als:SST}.
\end{proof}

\begin{remark}
	Let $I = {\rm Ann}_D(f)$ be the annihilator of a holonomic function~$f$,
	and fix a point $x_0 \in \mathbb{C} ^n$ that is not in the singular locus of~$I$.
	Let $m_1,\ldots,m_n$ be the orders of  the distinguished operators 
	$P_1,\ldots,P_n \in I$  in Proposition \ref{als:holon}~c).
	Thus, $P_k$ is a differential operator in $\partial_k$ of order $m_k$
	whose coefficients are polynomials in $x_1,\ldots,x_n$.
	Suppose we  impose initial conditions
	by specifying complex numbers for the
	$m_1m_2 \cdots m_n$ quantities
	\begin{equation}
	\label{als:eqinitialcond} (\partial_1^{i_1}\cdots \partial_n^{i_n}\bullet f)\!\!\mid_{x=x_0} \quad
	\text{where} \,\,\,0 \leq i_k < m_k \,\text{ for }\, k=1,\ldots,n.
	\end{equation}
	The operators $P_1,\ldots,P_n$ together with the initial conditions
	\eqref{als:eqinitialcond} determine the function $f$ uniquely within the vector space ${\rm Sol}(I)$.
	This specification is known as a \emph{canonical holonomic representation}
	of $f$; see \cite[Section~4.1]{als:Zei}.
\end{remark}

Many interesting functions are holonomic. To begin with,
every rational function $r$ in $x_1,\ldots,x_n$ is holonomic.
This follows from Proposition~\ref{als:holon} c), since $r$ is annihilated by the operators
\begin{equation}
\label{als:eqratop1} r(x) \partial_i \,-\,\frac{\partial r}{\partial x_i} \,\in \, R \qquad
\hbox{for} \,\,\,\, i = 1,2,\ldots,n .
\end{equation}
By clearing denominators in such a  first-order operator, we obtain a non-zero operator
$P_i\in \mathbb{C}  [x]\langle \partial_i \rangle$ with $m_i=1$ that annihilates $r$.
The operators $P_i$, together with fixing the value $r(x_0)$ at a general
point $x_0 \in \mathbb{C} ^n$, constitute a canonical holonomic representation.
If $r(x)$ is rational, then the function $g(x) = {\rm exp}(r(x))$
is also holonomic. The role of \eqref{als:eqratop1} is now played~by
\begin{equation}
\partial_i \,-\,\frac{\partial r}{\partial x_i} \,\in \, R \qquad
\hbox{for} \,\, i = 1,2,\ldots,n .
\end{equation}
By the Chain Rule, these first-order operators annihilate  $\,g(x) = {\rm exp}(r(x))$.

Holonomic functions in one variable are solutions to ordinary linear differential
equations with rational function coefficients.
Examples include algebraic functions, some elementary trigonometric functions, 
hypergeometric functions, Bessel functions, periods (as in Definition \ref{als:defperiod}),
and many more.
But, not every nice function is holonomic. A necessary condition for
a meromorphic function to be holonomic is that it has only finitely 
many poles in the complex plane. The reason is that the
singular locus of a holonomic $D$-ideal is an algebraic variety in~$\mathbb{C} ^n$.
Thus, for $n=1$ the singular locus must be a finite subset of~$\mathbb{C} $.

For a concrete example, the meromorphic function $f(x) = \frac{1}{\sin (x)}$ is not holonomic.
This shows that the class of holonomic functions is not closed under division, since
$\,\sin(x) \in {\rm Sol}(\langle \partial^2 + 1 \rangle)$.
It is also  not closed under composition of functions, since 
both $\frac{1}{x}$ and $\sin(x)$ are holonomic. 
But, as a partial rescue, following~\cite[Theorem 2.7]{als:Stanley}, 
we record the following positive result.

\begin{alsproposition}
	Let $f(x)$ be holonomic and $g(x)$ algebraic. Then their composition $f(g(x))$ is
	a holonomic function.
\end{alsproposition}

\begin{proof}
	Let $h\coloneqq f\circ g$. By the Chain Rule, all derivatives $h^{(i)}$ can 
	be expressed as linear combinations of $f(g),f'(g),f''(g),\ldots$ with coefficients 
	in $\mathbb{C}  [g,g',g'',\ldots]$. Since $g$ is algebraic, it fulfills some polynomial 
	equation $G(g,x)=0$. By taking derivatives of this equation,
	 we can express
	each $g^{(i)}$ as a rational function of $x$ and $g$. We conclude that  the ring
	$\mathbb{C}  [g,g',\ldots] $ is contained in the field $ \mathbb{C} (x,g)$. Denote by $W$ the vector space 
	spanned by $f(g),f'(g),\ldots$ over $\mathbb{C} (x,g)$ and by~$V$ the vector 
	space spanned by $f,f',\ldots$ over $\mathbb{C} (x)$. Since $f$ is holonomic, 
	$V$ is finite-dimensional over $\mathbb{C} (x)$. 
	This implies that $W$ is finite-dimensional over $\mathbb{C} (x,g)$. Since $g$ is algebraic, 
	$\mathbb{C} (x,g)$ is finite-dimensional over $\mathbb{C} (x)$. It follows that $W$ is a 
	finite-dimensional vector space over $\mathbb{C} (x)$, hence
	$h=f\circ g$ is holonomic.
\end{proof}

The term ``holonomic function'' was first proposed by D.~Zeilberger \cite{als:Zei}
in the context of proving combinatorial identities. Building on Zeilberger's work, 
among others, C.~Koutschan~\cite{als:CK} developed practical 
algorithms for manipulating holonomic functions.  These are implemented in his
{\tt {\tt Mathematica}} package {\tt HolonomicFunctions}, as seen below.

\begin{ex}\label{als:exmakeunique}
	Every algebraic function $f(x)$ is holonomic. Consider the function
	$y= f(x) $ that is defined by $y^4+x^4+\frac{xy}{100}	-1 \, = \,0$.
	Its annihilator in $D$ can be computed in {\tt Mathematica} as follows:
	\begin{verbatim}
	<< RISC`HolonomicFunctions`
	q = y^4 + x^4 + x*y/100 - 1
	ann = Annihilator[Root[q, y, 1], Der[x]]
	\end{verbatim}
	This {\tt Mathematica} code determines an operator $P$
	of lowest order in $\text{Ann} _D(f)$:
	\[ \begin{small} \begin{matrix}
	P \, \,=\, \,(2 x^4{+}1)^2 (25600000000 x^{12}{-}76800000000 x^8{+}76799999973 x^4{-}25600000000) \,\partial^3 \quad \\
	+ 6 x^3 (2 x^4{+}1) (51200000000 x^{12}{+}76800000000 x^8{-}307199999946x^4{+}
	179199999973) \,\partial^2 \\
	+ \,\,3 x^2 (102400000000 x^{16}{+}204800000000 x^{12}{+} 
	2892799999572 x^8 
	{-} 3507199999444x^4 \\ + 307199999953) \,\partial 
	\,\,-\,\,3x(102400000000x^{16} +204800000000 x^{12} \quad \\  \  +1459199999796 x^8
	-1049599999828x^4+51199999993).
	\end{matrix} \end{small} \]
	This operator encodes the algebraic function
	$y=f(x)$ as a holonomic function.
\end{ex}

In computer algebra, one represents a real algebraic
number as a root of a polynomial with coefficients in $\mathbb{Q}$.
However, this {\em minimal polynomial} does not specify the number uniquely. 
For that, one also needs an isolating interval or sign conditions on derivatives.
The situation is analogous when we encode a holonomic function $f$ in 
$n$ variables. 
We specify $f$ by a holonomic system of linear PDEs
together with a list of initial conditions.
The canonical holonomic representation is one example.
Initial conditions such as \eqref{als:eqinitialcond}
are designed to determine
the function uniquely inside the linear space ${\rm Sol}(I)$,
where $I \subseteq {\rm Ann}_D(f)$. For instance,
in Example~\ref{als:exmakeunique}, we would need three initial conditions to
specify the function $f(x)$ uniquely inside the $3$-dimensional solution
space to our operator~$P$. We could fix the values at three
distinct points, or we could
fix the value and the first two derivatives at one special point.

To be more precise, we generalize the canonical representation
\eqref{als:eqinitialcond} as follows.
A {\em holonomic representation} of a function $f$
is a holonomic $D$-ideal $I\subseteq \text{Ann} _D\left(f\right)$
together with a list of linear conditions that specify $p$ uniquely
inside the finite-dimensional solution space of holomorphic solutions.
The existence of this representation makes $f$ a {\em holonomic function}.
Before discussing more of the basic theory of holonomic functions,
notably their remarkable closure properties,
we first present an example that justifies the title of this~lecture.

\begin{ex}[The area of a TV screen]\label{als:areaTV}
	Let 
	\begin{align}\label{als:qtv}
	q(x,y) \,\,=\,\, x^4+y^4+\frac{1}{100}xy -1.
	\end{align} 
	We are interested in the semi-algebraic set $S = \{ (x,y)\in \mathbb{R} ^2 \mid q(x,y) \leq 0 \}$. 
	This convex set is a slight modification of a set known in the optimization literature as 
	``the TV screen''. Our aim is to compute the area of the semi-algebraic convex 
	set $S$ as accurately as is possible.
	
	One can get a rough idea of the area of $S$ by sampling. This is illustrated in
	Figure~\ref{als:figTVscreen}. From the equation we find that  $ S$ is contained in 
	the square defined by $-1.2 \leq x,y \leq 1.2$.
	We sampled $10000$ points uniformly  from that square, and for each sample
	we checked the sign of $q$. Points inside $S$ are drawn in blue
	and points outside $S$ are drawn in pink. By multiplying the area 
	$(2.4)^2 = 5.76$ of the square with the fraction of the number of blue points 
	among the samples, we learn that
	the area of the TV screen is approximately $3.7077$.
	
	\begin{figure}[h]
		\begin{center}
			\includegraphics[width=6cm]{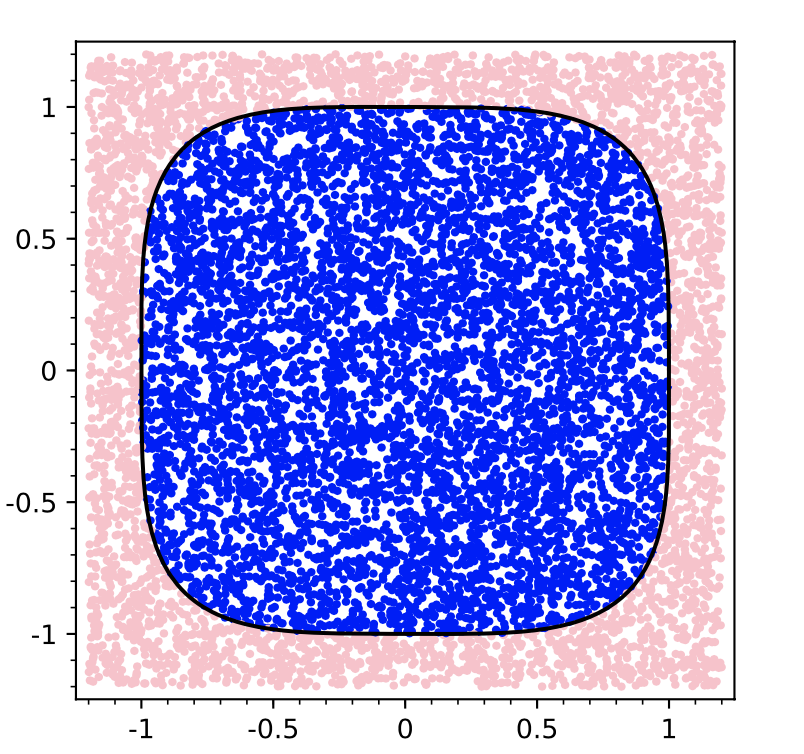} \vspace{-0.15in}
		\end{center}
		\caption{The TV screen is the convex region consisting of the blue points.
			\label{als:figTVscreen}}
	\end{figure}
	
	We now compute the area more accurately using $D$-modules. Let
	$\text{pr}\colon S \to \mathbb{R}  $ be the projection on the $x$-coordinate, 
	and write $v(x)= \ell \left( {\rm pr}^{-1}(x)\cap S \right)$
	for the length of a fiber.
	This function is holonomic and it satisfies the 
	third-order differential
	operator in Example~\ref{als:exmakeunique}.
	
	The map $\text{pr}$ has two branch points $x_0 < x_1$.
	They are the real roots of the~resultant ${\rm Res}_y( q, \partial{q}/ \partial{y}) $,
	which equals
	\begin{equation}
	\label{als:eqbranchpoints}
	25600000000 x^{12}-76800000000 x^8+76799999973x^4-25600000000 .
	\end{equation}
	These values can be written in radicals, but we take an accurate floating point
	representation. The branch point $x_1= -x_0$ is equal to
	\[  
	1.000254465850258845478545766643566750080196276158976351763236
	\ldots \]
	The desired area equals ${\rm vol}(S) = w(x_1)$, where $w$ is the holonomic function
	\[ w(x) \,\, = \,\, \int_{x_0}^x v(t) dt . \]
	One operator that annihilates $w$ is $ P \partial$,
	where $P \in \text{Ann} _D(v)$ is the third-order operator above.
	To get a  holonomic representation of $w$, we also need some initial conditions.
	Clearly, $w(x_0)=0$. Further initial conditions on $w'$ are derived by evaluating $v$ at other points. 
	By plugging values for $x$ into~\eqref{als:qtv} 
	and solving for $y$, we find $w'(0)=2$ and $w'(\pm 1) = 1/\sqrt[3]{100}$. 
	Thus, we now have four linear constraints on our function $w$, 
	albeit at different points.
	
	Our goal is to determine a unique function $w \in {\rm Sol}(P\partial)$
	by incorporating these four initial conditions, and then to
	evaluate  $w$ at $x_1$. To this end,  we proceed as follows.
	Let $x_{\text{ord}} \in \mathbb{R} $ be any point at which $P \partial$ is not singular.
	Using the command {\tt local\_basis\_expansion} 
	that is built into the 
	{\tt Sage} package {\tt ore\_algebra} \cite{als:JJK}, we compute a basis
	of local series solutions to $P\partial$ at the point $x_{\text{ord}}$.
	Since that point is non-singular, there exists a basis of the following form:
	\begin{equation}
	\label{als:eqlocalbasis}
	\begin{matrix}
	s_{x_\text{ord},0} (x)& = & 1 & + & O((x-x_{\text{ord}})^4), \\
	s_{x_\text{ord},1} (x)& = &  \! (x-x_{\text{ord}})\, & + & O((x-x_{\text{ord}})^4),\\ 
	s_{x_\text{ord},2} (x)& = & (x-x_{\text{ord}})^2 & + & O((x-x_{\text{ord}})^4),\\ 
	s_{x_\text{ord},3}(x) & = &  (x-x_{\text{ord}})^3 & + & O((x-x_{\text{ord}})^4).
	\end{matrix}
	\end{equation}
	Indeed,  by applying Proposition \ref{als:propdistrac} locally at $x_{\rm ord}$, we obtain
	the initial ideal
	$\,\text{in} _{(-1,1)}(\langle P \partial \rangle) = \langle \partial^4 \rangle$.
	By Theorem~\ref{als:thmdistrac}, this
	$D$-ideal has the distraction $J = \langle \theta (\theta-1) (\theta-2)(\theta-3)\rangle$.
	Its variety equals $\,V(J) = \{0,1,2,3\}$.
	Locally at  $x_{\text{ord}}$, our solution is given by a unique choice of
	four coefficients $c_{x_{\text{ord}},i}$, namely
	\[ w(x) = c_{x_{\rm ord},0} \cdot s_{x_\text{ord},0} (x) \,+\,
	c_{x_{\rm ord},1} \cdot s_{x_\text{ord},1} (x) \,+\,
	c_{x_{\rm ord},2} \cdot s_{x_\text{ord},2} (x) \,+\,
	c_{x_{\rm ord},3} \cdot s_{x_\text{ord},3} (x) .\]
	Let us point out that at a regular singular point $x_{\text{rs}}$,  both
	complex powers of $x$ and 
	positive integer powers of $\log(x)$ can be involved in the local basis extension at $x_{\text{rs}}$.
	We saw  a series with $\log(x)$ in Example~\ref{als:exwithlog}.
	
	Any initial condition at that point determines a linear constraint on these coefficients.
	For instance, $w'(0) =2$ implies $c_{0,1}=2$, and similarly for our initial conditions at
	$-1,1$ and~$x_0$. One challenge we encounter here is that the initial conditions pertain 
	to different points.
	To address this, we calculate transition matrices that
	relate the basis \eqref{als:eqlocalbasis} of series solutions at one 
	point to the basis of series solutions at another point. 
	These are invertible $4 \times 4$ matrices. 
	We compute them in {\tt Sage} with the command {\tt op.numerical\_transition\_matrix}.
	
	With the method described above, we find the basis of series solutions at 
	$x_1$, along with a system of four linear constraints on the
	four coefficients $c_{x_1,i}$. These constraints are derived
	from the initial conditions at $0,\ \pm 1$, and $x_0$, 
	using the $4 \times 4$ transition matrices. By solving these
	linear equations, we compute the desired function value $w(x_1)$ up to any desired precision:
	\[ 
	3.708159944742162288348225561145865371243065819913934709438572
	.... \]
	In conclusion, this positive real number is the area of the TV screen $S$
	defined by the polynomial $q(x,y)$. All of the listed digits are expected to be correct.
\end{ex}
Let us now come back to properties of holonomic functions. 
Holonomic functions are very well-behaved with respect to many operations. 
They turn out to have remarkable closure properties. In the following, 
let $f,g$ be functions in $n$ variables $x_1,\ldots,x_n$. 
In order to prove that a function is holonomic, we use one of the equivalent
characterizations in Proposition~\ref{als:holon}.

\begin{alsproposition}\label{als:prodhol}
	If $f,g$ are holonomic functions, then both 
	their sum $f+g$ and their product $f\cdot g$ are holonomic functions as well.
\end{alsproposition}

\begin{proof}
	For each $i \in \{1,2,\ldots,n\}$, there exist non-zero operators
	$P_i, Q_i \in \mathbb{C}  [x]\langle \partial_i \rangle$, such that
	$P_i \bullet f = Q_i \bullet g = 0$. Set $n_i = {\rm order}(P_i)$ and
	$m_i = {\rm order}(Q_i)$. The
	$\mathbb{C} (x)$-span of $\left\{\partial_i^k \bullet f \right\}_{k=0,\ldots,n_i}$ 
	is a vector space of dimension $\leq n_i$. 
	Similarly, the $\mathbb{C} (x)$-span of the set 
	$\left\{\partial_i^k \bullet g \right\}_{k=0,\ldots,m_i}$ 
	has dimension $\leq m_i$.
	
	Now consider 
	$\partial_i^k \bullet \left( f+g \right)=\partial_i^k \bullet f+\partial_i^k \bullet g$.
	The $\mathbb{C} (x)$-span of\linebreak
	$\left\{ \partial_i^k \bullet \left( f+g \right) \right\}_{k=0,\ldots,n_i+m_i}$
	has dimension $\leq n_i+m_i$. Hence, there exists a non-zero
	operator $S_i\in \mathbb{C} [x]\langle \partial_i \rangle$, such that $S_i\bullet(f+g)=0$. 
	Since this holds for all indices $i$, we conclude that the sum  $f+g$ is holonomic.
	
	A similar proof works for the product $f \cdot g$. For each $i \in \{1,2,\ldots,n\}$,
	we now consider the set 
	$\left\{ \partial_i^k \bullet \left( f\cdot g\right) \right\}_{k=0,1,\ldots,n_im_i}$.
	By applying Leibniz' rule for taking derivatives of a product, 
	we find that the	$m_in_i+1$ generators are linearly dependent over
	the field $\mathbb{C} (x)$.
	Hence, there is a non-zero
	operator $T_i\in \mathbb{C} [x]\langle \partial_i \rangle$ such that $T_i\bullet(f \cdot g)=0$. 
	We conclude that $f \cdot g$ is holonomic.
\end{proof}
\begin{remark} The proof above gives a linear algebra method for computing
	an annihilating $D$-ideal $I$ of finite holonomic rank for $f+g$ (resp. of $f\cdot g$), 
	starting from such  annihilating $D$-ideals for $f$ and $g$. More refined methods for the same task 
	can be found in \cite[Section 3]{als:Tak92}. To get an ideal that is actually holonomic, it may
	be necessary to replace $I$ by its Weyl closure $W(I)$.
\end{remark}

The following example
illustrates Proposition~\ref{als:prodhol}.

\begin{ex}[$n=1$]
Following~\cite[Section 4.1]{als:Zei}, we
	consider the functions $f(x)=\exp (x)$ and $g(x)=\exp (-x^2)$.
	Their canonical holonomic representations are
	$I_f = \langle \partial - 1\rangle$ with $f(0) = 1$ and
	$I_g = \langle \partial + 2x \rangle $ with $g(0) = 1$.
	We are interested in the function $h = f+g$. 
	We write
	\[ \begin{pmatrix} 
	h \\ \partial \bullet h \\ \partial^2 \bullet h \end{pmatrix} \,\,\, = \,\,\,
	\begin{pmatrix} 1 & 1 \\
	1 & -2x \\
	1 & 4 x^2 -2 \end{pmatrix} \cdot
	\begin{pmatrix}
	f \\ g
	\end{pmatrix}.\]
	By computing the left kernel of the $3 \times 2$ matrix on the right hand side, we find that
	$h = f+g$ is annihilated by
	\[ I_h \,\,=\,\, \langle
	(2x+1) \partial^2 + 
	(4x^2-3) \partial
	-4x^2-2x+2\rangle , \quad \hbox{with}\,\, \,h(0) = 2, \,h'(0) = 1. \]
	Similarly, for the product $\,j = f \cdot g\,$ we find that
	$$ \,j'\, =\, f'g + f g' \,\,=\,\, f\cdot g + f \cdot(- 2xg) \,\,= \,\,(1-2x) j. $$
	A canonical holonomic representation of $j$ is 
	$I_j = \langle \partial + 2x - 1 \rangle$ with~$j(0) = 1$.
\end{ex}

\begin{alsproposition}\label{als:proprestrict}
	Let $f$ be a holonomic function in $n$ variables and $m < n$.
	Then the restriction of $f$ to the coordinate subspace 
	$\{x_{m+1}=\ldots=x_n=0\}$
	is a holonomic function in the $m$ variables $\,x_1,\ldots,x_m$.
\end{alsproposition} 

\begin{proof} For $i \in \{m+1,\ldots,n\}$,  we consider the right ideal $x_iD_n$ 
	in the Weyl algebra $D_n$. This ideal is a left module over 
	$D_m= \mathbb{C}  \left[ x_1,\ldots,x_m\right] \langle \partial_1,\ldots,\partial_m\rangle$.
	The sum of these ideals with ${\rm Ann}_{D_n}(f)$ is hence a left $D_m$-module.
	Its intersection with $D_m$ is called the {\em restriction ideal}:
	\begin{equation}\label{als:eqrestrictionideal}
	(\text{Ann} _{D_n}(f)+x_{m+1}D_n+\cdots+x_nD_n ) \,\,\cap \,\,D_m .
	\end{equation}
	By \cite[Proposition~5.2.4]{als:SST}, this $D_m$-ideal is holonomic and it
	annihilates the restricted function $f(x_1,\ldots,x_m,0,\ldots,0)$.
\end{proof}

\begin{alsproposition}\label{als:proppartial}
	The partial derivatives of a holonomic function are holonomic functions.
\end{alsproposition}

\begin{proof}
	Let $f$ be holonomic and $P_i\in \mathbb{C} [x]\langle \partial_i\rangle\backslash \{0\}$ 
	with $P_i\bullet f=0$ for all $i$.  We can write
	$P_i$ as $P_i=\widetilde{P_i}\partial_i+a_i(x)$, where $a_i\in\mathbb{C} [x]$. 
	If $a_i=0$, then $\widetilde{P_i}\bullet \frac{\partial f}{\partial x_i} = 0$ 
	and we are done.
	Assume $a_i\neq 0$. Since both $a_i$ and $f$ are holonomic, 
	by Proposition~\ref{als:prodhol}, there is a non-zero linear operator 
	$Q_i\in \mathbb{C} [x]\langle \partial_i \rangle$ such that 
	$Q_i \bullet (a_i\cdot f)=0$.   
	Then $Q_i\widetilde{P_i}$ annihilates $ \partial f / \partial x_i$.
\end{proof}

The Fundamental Theorem of Calculus ensures that
indefinite integrals of holonomic functions are holonomic functions.
In order to prove a similar statement for the case of definite integrals, 
one has to work a little harder.

\begin{alsproposition}\label{als:inthol}
	For a holonomic function $f\colon \mathbb{R}  ^n \to \mathbb{C} $,
		the definite integral 
	\[ F(x_1,\ldots,x_{n-1}) \,\,=\,\,\int_a^b f(x_1,\ldots,x_{n-1},x_n)dx_n\]
	is a holonomic function in $n-1$ variables, assuming the integral exists.
\end{alsproposition}

\begin{proof} We consider the case $n=2$, the general case being proven in the very same manner. 
	Let $D_2=\mathbb{C}[x_1,x_2]\langle \partial_1,\partial_2 \rangle$ and $I$ a holonomic $D_2$-ideal annihilating $f$.
	Since $I$ is holonomic, its {\em integration ideal} 
	$I_{\text{int}} \coloneqq (I+\partial_2D_2 ) \cap \mathbb{C}[x_1]\langle \partial_1 \rangle$
	 is a holonomic $ \mathbb{C}[x_1]\langle \partial_1 \rangle$-ideal (cf. \cite[Theorem 6.10.3]{als:Tak}). 
	 Therefore, there exists a non-zero operator $Q\in I_{\text{int}}$.
	 Moreover, $Q$ can be written as $Q=P-\partial_2R$ 
	 for some $P\in I\setminus \{0\}$ and $R\in D_2$. 
	 In particular, $P=Q(x_1,\partial_1)+\partial_2R(x_1,x_2,\partial_1,\partial_2)$ annihilates $f$.
	Applying this operator to $f$ and taking the integral $\int_a^b$ on both sides yields 
	\begin{align}\label{als:inteq}
	0 \,\,= \,\, Q(x_1,\partial_1) \bullet F(x_1) \, +\,
	\bigl[ R(x_1,x_2,\partial_1,\partial_2)\bullet f \bigr]_{x_2=a}^{x_2=b}.
	\end{align}
	If the second summand in~\eqref{als:inteq} is zero, then 
	$Q\in \mathbb{C} [x_1]\langle \partial_1 \rangle$ annihilates $F(x_1)$
	and we are done. Otherwise, that summand is a holonomic 
	function in $x_1$ by Propositions \ref{als:proprestrict} and \ref{als:proppartial}.
	Thus, there exists a non-zero operator $\widetilde{Q}\in \mathbb{C} [x_1]\langle \partial_1 \rangle$ 
	that annihilates $\left[ R(x_1,x_2,\partial_1,\partial_2)\bullet f \right]_{x_2=a}^{x_2=b}$.
 	Then, applying the same argument to $\widetilde{Q}P$ concludes the proof.
	
	Here is an alternative second proof. We can use the Heaviside function
	\[ H(x)\,\,=\,\,\begin{cases}  0 & x<0, \\ 1 & x \geq 0, \end{cases}\] 
	to rewrite the integral as follows: 
	\[F(x_1) \,\, = \,\,\int_a^b \!\! f(x_1,x_2)dx_2 \,\,=
	\,\, \int_{-\infty}^{\infty} \!\! H(x_2-a)H(b-x_2) f(x_1,x_2)dx_2.\]
	The distributional derivative of $H$ is given by the Dirac delta function $\delta$. 
	Since $x \bullet \delta =0$, the operator $\theta = x\partial$ annihilates the Heaviside function, 
	so $H$ is a holonomic distribution. Adapting the above proof to the 
	holonomic distribution $H(x_2-a)H(b-x_2) f(x_1,x_2)$, 
	the second summand in~\eqref{als:inteq} vanishes, since 
	\mbox{$H(x_2-a)H(b-x_2) f(x_1,x_2)$} is supported on $[a,b]$.    
\end{proof}

\begin{defn}\label{als:defintegrationideal}
	As in the proof of Proposition~\ref{als:inthol}, it is natural to consider 
	\[ \qquad \bigl( {\rm Ann}_{D_n}(f) \,\,+\,\,\partial_{m+1}D_n + \cdots + \partial_nD_n \bigr) 
	\,\cap \, D_m 
	\qquad \text{for }\,\, m< n. \]
	This intersection is a $D_m$-ideal.
	It is called the \emph{integration ideal} of the function~$f$ with respect to 
	the variables $x_{m+1},\ldots,x_n$.
	The expression
	is dual to the restriction ideal \eqref{als:eqrestrictionideal} under the Fourier transform.
 	The Fourier transform exchanges $x_i$ and~$\partial_i$, with a minus sign involved.
\end{defn}

Equipped with our tools for manipulating holonomic functions, 
we now embark on the computation of volumes of compact semi-algebraic sets. 
We follow the work of P.~Lairez, M.~Mezzarobba, and M.~Safey El Din~\cite{als:LMS}. 
They compute this volume by deriving the Picard--Fuchs differential equation of 
the period of a certain rational integral. Here is the key definition.

\begin{defn}\label{als:defperiod} \rm
For a rational function
	 $R(t,x_1,\ldots,x_n)$, consider the  integral
	\begin{equation}
	\label{als:eqperiod1}
	\oint R(t,x_1,\ldots,x_n)dx_1\cdots dx_n.
	\end{equation}
	We also fix an open subset $\Omega$ of either $\mathbb{R} $ or $\mathbb{C} $.
	An analytic function $\phi \colon \Omega \to \mathbb{C} $ is a 
	{\em period} of the integral \eqref{als:eqperiod1}
	if, for any $s\in \Omega$, there exists a 
	neighborhood $\Omega' \subseteq \Omega$ of $s$ and
	an $n$-cycle $\gamma \subset \mathbb{C} ^n$ with the following property.
	For all $t\in \Omega'$, the cycle $\gamma$ is disjoint from the poles of 
	$R_t \coloneqq R(t,\bullet)$ and
	\begin{equation}
	\phi (t) \,\,=\,\, \int_{\gamma} R(t,x_1,\ldots,x_n)dx_1\cdots dx_n . 
	\end{equation}   
	If this holds, then there exists an operator
	$P\in D \backslash \{0\}$ of the Fuchsian class that annihilates~$\phi(t)$. 
\end{defn}

Let $S = \left\{ f \leq 0 \right\} \subset \mathbb{R}  ^n$ be a compact basic 
semi-algebraic set, defined by a polynomial $f\in \mathbb{Q} [x_1,\ldots,x_n]$. 
Let $\text{pr}\colon \mathbb{R}  ^n \to \mathbb{R}  $ denote the projection 
onto the first coordinate.
The set of  {\em branch points} of
the hypersurface $\{f=0\}$ under the map $\text{pr}$ is the following subset 
of the real line, which is assumed to be finite:
\begin{align*} 
 \Sigma_f \,\,=\,\, \left\{ \right. p \in \mathbb{R}  \, \mid \ & \exists
\,x=(x_2,\ldots,x_n) \in \mathbb{R}  ^{n-1}: f(p,x)=0 \,\,\,
\text{and} \\ 
&\frac{\partial f}{\partial x_i}(p,x)=0\, \text{ for } i=2,\ldots,n\left.\right\}.
\end{align*}
The polynomial in the unknown $p$ that defines $\Sigma_f$ is obtained by
eliminating $x_2,\ldots,x_n$. It can be represented as a multivariate resultant, 
generalizing the Sylvester resultant in \eqref{als:eqbranchpoints}.

Fix an open interval $I$ in $\mathbb{R} $ with $I \cap \Sigma_f = \emptyset$.
For any $x_1 \in I$, the set\linebreak$S_{x_1} \coloneqq \text{pr}^{-1}(x_1)\cap S \,$
is compact and semi-algebraic in $(n-1)$-space. We are
interested in the volume of this set. By \cite[Theorem 9]{als:LMS}, 
the function  $\,v\colon I\to \mathbb{R} ,\,x_1 \mapsto \text{vol}_{n-1} \left( S_{x_1}\right)\,$ 
is a period of the rational integral
\begin{align}\label{als:volumefiber}
\frac{1}{2\pi i} \oint  \frac{x_2}{f(x_1,x_2,\ldots,x_n)}\frac{\partial f(x_1,x_2,\ldots,x_n)}{\partial x_2} dx_2 \cdots dx_n.
\end{align} 

Let $e_1 < e_2 < \cdots < e_K$ be the branch points in $\Sigma_f$
and set $e_0 = - \infty$ and $e_{K+1}=\infty$.
This specifies the pairwise disjoint
open intervals $I_k=\left(e_k,e_{k+1} \right)$. 
They satisfy $\,\mathbb{R}  \backslash \Sigma_f = \bigcup_{k=0}^{K} I_k$.
Fix the holonomic functions $\,w_k(t) = \int_{e_k}^t v(x_1)dx_1$. 
The volume of $S$  is obtained as
\[ \text{vol}_n \left( S \right) 
\,\, = \,\, \int_{e_1}^{e_K} \!\! v(x_1) d x_1 \,\,=\,\,
\sum_{k=1}^{K-1} w_k\left( e_{k+1}\right).\]

How does one evaluate such an expression numerically?
As a period of the rational integral~\eqref{als:volumefiber}, 
the volume function $v$
is a holonomic function on each interval $I_k$. A key step 
is to compute a differential operator $P \in D_1$ that annihilates $v\!\!\mid_{I_k}$ for all $k$.
With this, the product operator $P\partial$ annihilates the function
$w_k(x_1)$ for $k=1,\ldots,K-1$. By imposing sufficiently many initial conditions, 
we can reconstruct the functions $w_k$ 
from the  operator $P\partial$ uniquely. One initial condition that comes for free 
for each~$k$ is $w_k(e_k)=0$.

\smallskip

The operator $P$ is known as the {\em Picard--Fuchs equation} 
of the period in question.
The following software packages can  compute such Picard--Fuchs equations:
\begin{bulist}
	\item {\tt HolonomicFunctions}~\cite{als:CK} by C.~Koutschan in {\tt Mathematica}, 
	\vspace{-0.2cm}
	\item {\tt Ore\_algebra} by F.~Chyzak in {\tt Maple},
		\vspace{-0.2cm}
			\item {\tt ore\_algebra}~\cite{als:JJK}  by M.~Kauers in {\tt Sage},
		\vspace{-0.2cm}
	\item {\tt Periods} by P.~Lairez in {\tt Magma}, implementing the algorithm described in~\cite{als:L}.
\end{bulist}
Our readers are encouraged to experiment with these programs.

We next discuss how one can actually compute 
the volume of our semi-algebraic set $S = \{f \leq 0\}$ in practice.
Starting from the defining polynomial $f$,
we compute the Picard--Fuchs operator $P \in D_1$ 
and we find sufficiently many compatible initial conditions. 
Thereafter, for each interval $I_k$, where $k=1,\ldots,K-1$,
we perform the following steps. We describe this for
the \tt{ore\_algebra} \rm package in {\tt Sage}, which 
we found to work well:
\begin{numlist}
	\item Using the command {\tt local\_basis\_expansion}, compute 
	a local basis of series solutions for the 
	differential operator $P \partial $ at various points in $[e_k,e_{k+1})$.
	\item Using the command {\tt op.numerical\_transition\_matrix}, 
	compute a transition matrix for the series solution basis from one point to 
	another one.
	\item From the initial conditions, construct linear relations between the 
	coefficients in the local basis extensions. Using step 2, transfer them to 
	the branch point $e_{k+1}$.
	\item Plug in to the local basis extension at $e_{k+1}$ and thus
	evaluate the volume of~$S\cap \text{pr}^{-1}\left( I_k\right)$.
\end{numlist}

We illustrate this recipe by computing the volume of a 
convex body in $3$-space.

\begin{ex}[Quartic surface]
	Fix the quartic polynomial
	\begin{align}
	f(x,y,z)\,\,=\,\,x^4+y^4+z^4+\frac{x^3y}{20}-\frac{xyz}{20}-\frac{yz}{100}+\frac{z^2}{50} -1,
	\end{align}
	and let $\,S= \left\{(x,y,z)\in \mathbb{R}^3 \mid f(x,y,z) \leq 0 \right\}$. 
	Our aim is to compute $\text{vol}_3\left( S \right)$.  
	\begin{figure}[h]
		\begin{center}
			\includegraphics[width=6cm]{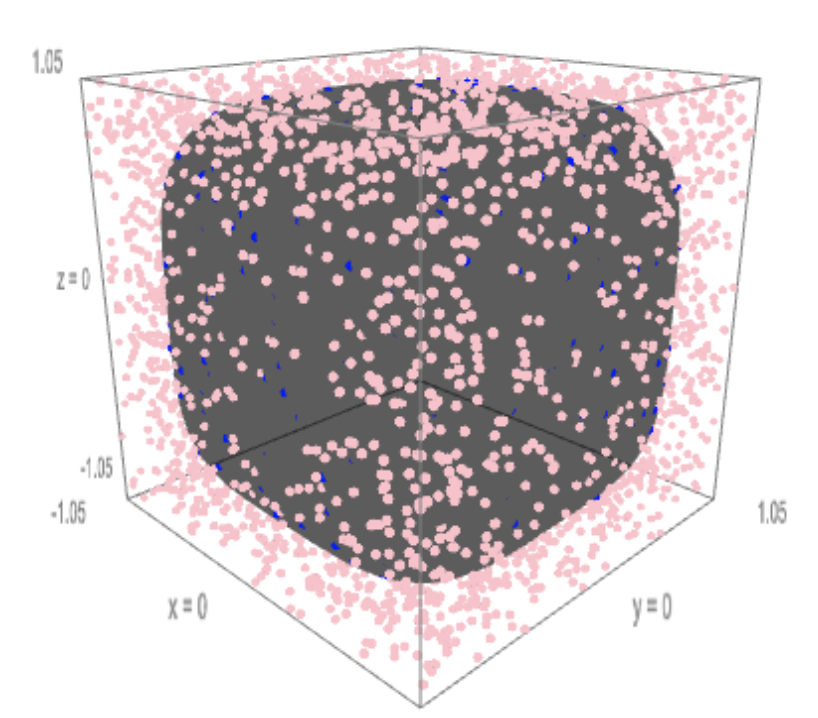} \vspace{-0.19in}
		\end{center}
		\caption{The quartic bounds the convex region consisting of the gray points.
			\label{als:figquartic}}
	\end{figure}
	
	As in Example~\ref{als:areaTV} with the TV screen, we
	can get a rough idea of the volume of~$S$ by sampling. 
	This is illustrated in Figure~\ref{als:figquartic}. Our set $S$ is compact, convex,
	and contained in the cube given by $-1.05 \leq x,y,z \leq 1.05$. 
	We sampled $10000$ points uniformly from that cube. 
	For each sample we checked the sign of $f(x,y,z)$. 
	By multiplying the cube's volume $(2.1)^3=9.261$ by the fraction~of 
	the number of gray points and the number of sampled points, 
	{\tt Sage} found within few seconds that the volume of the quartic
	body $S$ is~$\approx  6.4771$. 
	In order to obtain a higher precision, we now compute the volume 
	of $S$ with the help of $D$-modules.
	
	We use the notation
	 $\text{pr}\colon \mathbb{R}  ^3 \to \mathbb{R}  $ for the projection onto the $x$-coordinate.
	Let \linebreak$v(x)= \text{vol}_2 \left( {\rm pr}^{-1}(x)\cap S \right)$ denote the area of the 
	fiber over any point $x$ in~$\mathbb{R} $. We write $e_1<e_2$ for the two branch points of
	the map $\text{pr}$ restricted to the quartic surface $\{f=0\}$.
	They can be computed by help of resultants, for instance by the
	following steps in {\tt Singular} with the library {\tt solve.lib}:
	\begin{verbatim}
	LIB "solve.lib";
	ring r=0,(x,y,z),dp; setring r;
	poly F=x^4+y^4+z^4+x^3*y*1/20-x*y*z*1/20-y*z*1/100+z^2*1/50-1;
	def DFy=diff(F,y); def DFz=diff(F,z);
	def resy=resultant(F,DFy,y); def resz=resultant(F,DFz,z);
	ideal I=F,resy,resz;
	def A=solve(I); setring A; SOL;
	\end{verbatim}
	The final output {\tt SOL} is a list of the $36$ roots of the zero-dimensional ideal $I$. 
	The first two of them are real. They are the branch points of
	$\text{pr}$. We obtain $e_1 \approx -1.0023512$ and $e_2\approx 1.0024985$.
		By \cite[Theorem~9]{als:LMS}, the area function $v(x)$ is a period of
	the rational integral 
	\[ \frac{1}{2 \pi i}  \oint \frac{y}{f(x,y,z)}\frac{\partial f(x,y,z)}{\partial y} dy dz.\] 
	Set $w(t) =\int_{e_1}^t v(x)dx$.
	The desired $3$-dimensional volume is $\text{vol}_3(S)=w(e_2)$.
	
	Using Lairez' implementation {\tt periods} in {\tt MAGMA}, we compute 
	a differential operator $P$ of order eight that annihilates $v(x)$.
	Again, $P \partial$ then annihilates $ w(x)$. One initial condition is $w(e_1)=0$. 
	We obtain eight further initial conditions $w'(x)=\text{vol}_2(S_{x})$ 
	for points $x \in (e_1,e_2)$ by running the same algorithm 
	for the $2$-dimensional semi-algebraic slices
	$S_{x} = {\rm pr}^{-1}(x) \cap S$. In other words, we make
	eight subroutine calls to an area measurement as in Example~\ref{als:areaTV}.
	
	From these nine initial conditions we derive linear relations of the coefficients 
	in the local basis expansion at $e_2$. These computations are run in {\tt Sage} as described
	in steps 1--4 above. We find the approximate volume of our convex body $S$ to be 
	\[ \begin{matrix}
	\approx
	6.43883248057289354474073389596995618895842088923511697626632892312\\8826 
	915527388764216209149558398903829431137608893452690352556009760\\1024171
	19080476940553482655811421276613538061395975793530527102208\\9419155701
	5215864701708740021943845291406868562277595417150971133\\9913473405961
	763289220607208551633239796916338376007073876010731\\8247752061504714
	36725046090092340906637773227339039682229623521\\4963623286613117557
	93068754414836072122568105348117876005826\\4738867105810326818911
	578448323758536767168707442532146029\\753762594261578920477859.
	\end{matrix}\]
	This numerical value is guaranteed to be accurate up to $550$ digits.
\end{ex}

\section{Statistics} 

In this lecture, we explore the role of $D$-modules in algebraic statistics.
Our discussion centers around two themes. First, we study the
Bernstein--Sato ideal of the likelihood function~of a discrete statistical model.
We present a case~study that suggests a relationship between that
ideal and maximum likelihood estimation. Thereafter, we turn to the
holonomic gradient method (HGM) for continuous distributions. This
method is the result of collaborations between statisticians and 
$D$-module experts in Japan. We explain how HGM is used
for maximum likelihood estimation.
This is implemented in an {\tt R} package~\cite{als:hgmR}.

Let $D_n=\mathbb{C} [x_1,\ldots,x_n]\langle \partial_1,\ldots,\partial_n \rangle$ be
the $n$th Weyl algebra and adjoin one new formal variable
$s$ that commutes with all $x_i$ and $\partial_j$. This defines the ring $D_n[s]$.
Fix a polynomial $f\in \mathbb{C} \left[ x_1,\ldots,x_n\right]$ and consider 
the \mbox{$D_n[s]$-module} $\mathbb{C} \left[x_1,\ldots,x_n,f^s,f^{-1},s\right]$.
Here the action of $D_n[s]$ on this module
is given by the usual rules of calculus and arithmetic, in particular
\[ \partial_i \bullet f^s \,\, = \,\, s\cdot \frac{\partial f}{\partial x_i} \cdot f^{-1} \cdot f^s
\,\, = \,\,  s\cdot  \frac{\partial f}{\partial x_i} \cdot f^{s-1} .\]

The {\em Bernstein--Sato polynomial}  of $f$ is the unique 
monic univariate polynomial 
$b_f\in\mathbb{C} [s]$ of minimal degree  such that 
$P\bullet f^{s+1}=b_f\cdot f^s$ for some $P\in D[s]$.
The polynomial $b_f$ was called the
{\em global $b$-function} in \cite[Section 5.3]{als:SST}, to which we refer for
details. It is known that $b_f$ is non-zero. M.~Kashiwara~\cite{als:Kas} showed that
all its roots are negative rational numbers.  
The Bernstein--Sato polynomial is computed
as the generator of the following principal ideal:
\[ \langle b_f\rangle \,\,=\,\, \left( \text{Ann}_{D[s]}\left( f^s \right) + \langle f \rangle \right) 
\, \cap\, \mathbb{C} [s].\]
Here, $\,\text{Ann}_{D[s]}\left( f^s \right) =\left\{ P\in D[s] \mid P \bullet f^s=0 \right\} \,$ 
denotes the $s$-parametric annihilator of $f^s$. A method for computing this 
$D[s]$-ideal can be found in \cite[Algorithm 5.3.15]{als:SST}.

We now pass to the case $k \geq 2$ of several polynomials
$f_1,\ldots,f_k\in \mathbb{C} \left[ x_1,\ldots,x_n \right]$. Following \cite{als:BVWZ} and the
references therein, we define the {\em Bernstein--Sato ideal} of 
the tuple of polynomials $f = (f_1,\ldots,f_k)$~as follows:
\begin{equation}\label{als:eqBSideal} 
\mathcal{B} \left( f \right) \,\,= \,\,\left(\, \text{Ann} _{D_n[s_1,\ldots ,s_k]}
\left( f_1^{s_1}\cdots f_k^{s_k} \, \right) \, +\, \langle f_1\cdots f_k \rangle \right) 
\, \cap \,\,\mathbb{C} [s_1,\ldots,s_k].
\end{equation}
Here, $s=(s_1,\ldots,s_k)$ and the
$s$-parametric annihilator of $f^s\coloneqq f_1^{s_1}\cdots f_k^{s_k}$~is
$$\,\text{Ann}_{D_n[s_1,\ldots,s_k]}\left( f^s \right) \,\, = \,\,
\bigl\{ P\in D_n[s_1,\ldots,s_k] \mid P \bullet f^s=0 \bigr\} . $$
An implementation for performing the computation
on the right hand side in~\eqref{als:eqBSideal}
is available in the {\tt Singular} library {\tt dmod.lib}.

Note that the ideal $\mathcal{B}(f)$ 
consists of all polynomials $b\in \mathbb{C}  [s_1,\ldots,s_k]$
that satisfy 
\[ \qquad b \cdot \prod_{i=1}^k f_i^{s_i} \,\,=\,\, P\bullet\prod_{i=1}^{k} f_i^{s_i+1} 
\quad  \hbox{for some $\,P\in D_n[s_1,\ldots,s_k]$.} \]
For $k=1$, the ideal $\mathcal{B}(f)$ is generated by the 
Bernstein--Sato polynomial $b_f$. 
If $k >1$, the Bernstein--Sato ideal is generally not principal. 
By \cite[Theorem 1.5.1]{als:BVWZ}, all the irreducible components
of codimension one in the variety of
$\mathcal{B}\left(f\right)$ are hyperplanes of the special form
\begin{align}\label{als:typehyperplane}
a_1s_1+\ldots+a_k s_k +b\,=0\, \quad
\text{ where }\,\,a_1,\ldots,a_k \in \mathbb{Q} _{\geq 0}
\,\,{\rm and } \,\, b\in \mathbb{Q} _{>0}.
\end{align}
This generalizes the fact that the roots of $b_f\in \mathbb{C} [s]$ 
are negative rationals.

We are interested in studying the Bernstein--Sato ideals
of parametric statistical models for discrete data.
These models are families of probability distributions on $k$ states,
where the probability of the $i$th state is given by a polynomial $f_i$.
The $n$ unknowns $ x = (x_1,\ldots,x_n)$ represent the parameters
of the models. A key feature of any statistical model is the identity
\begin{equation}
\label{als:eqthesumisone}  f_1(x) + f_2(x) + \cdots + f_k(x) \,\,=\, \, 1, 
\end{equation}
along with the following reasonable semi-algebraic hypothesis:
$$\,\exists \,u \in \mathbb{R} ^n \,\,\forall\, i \in \{1,\ldots,k\}\, : \,f_i(u) > 0. $$
The coefficients of $f_i$ are usually rational numbers.
Here is a familiar example:

\begin{ex}[Flipping a biased coin]\label{als:exflipping}
	Let $n=1$ and $m=k-1$, so we can reindex from $0$ to $m$.
	Set $f_i(x) = \binom{m}{i} x^i (1-x)^{m-i}\in \mathbb{Q} [x]$
	for $i=0,\ldots,m$. Identity \eqref{als:eqthesumisone} holds
	by the Binomial Theorem.   
	We also set $f^s \coloneqq f_0^{s_0} f_1^{s_1} \cdots f_m^{s_m}$.
	The unknown $x$ represents the {\em bias} of the coin,
	a real number between $0$ and~$1$. This is the probability
	that the coin comes up heads.
	Then $f_i(x)$ is the probability of observing
	$i$ heads among $m$ independent coin tosses.
	
	The $s$-parametric annihilator 
	$\,\text{Ann}_{D[s]}\left( f^s \right)\,$ 
	is the $D[s]$-ideal generated by 
	\[ x(x-1)\partial_x\,-\,m x\sum_{j=0}^{m}s_j\,+\,\sum_{j=2}^m(j-1)s_j. \]
	This is the result of a computation for small values of $m$.
	The Bernstein--Sato ideal 
	$\mathcal{B}\left(f\right)\subset \mathbb{Q} [s_0,\ldots,s_m]$ 
	is obtained by adding $f_0 f_1 \cdots f_m$ to this ideal 
	and then eliminating $x$ and $\partial$.
	We find that $\mathcal{B}\left(f\right)$ is the principal ideal generated 
	by the following product of linear forms:
	\begin{equation}
	\label{als:eqcoinsproduct} \!\!\!\!
	\prod_{j=1}^{\binom{m+1}{2}} \!\! \left( s_1+2s_2 + \cdots + m s_m + j \right) 
	\cdot \!\! \prod_{k=1}^{\binom{m+1}{2}} \!\! \left( m s_0 + (m{-}1)s_1+\cdots+s_{m-1}+k \right)  .
	\end{equation}
	One sees that all factors are linear forms with positive coefficients, as predicted in
	\eqref{als:typehyperplane}. Note that we can recover these linear factors in 
	a combinatorial way from the following table.
	The rows mean that $f_i$ is a monomial in $x$
	and $1-x$, and the columns specify the exponents of these monomials:
	\begin{center}
		\begin{tabular}{r|ccccc|l}
			&$f_0$ & $f_1$  &\ldots &$f_{m-1}$ & $f_m$ & $\sum$ \\ 
			\hline
			$x \quad $ & $0$ & $1$ & \ldots & $m-1$ & $m$ & $\binom{m+1}{2}$ \\
			$1-x\,$ & $m$&$m-1$&\ldots &$1$&$0$&$\binom{m+1}{2}$ 
		\end{tabular}
	\end{center}
\end{ex}

The validity of the Formula 
\eqref{als:eqcoinsproduct} can be derived from the results on hyperplane arrangements in
\cite[Section 6]{als:Bud}. The point is that, for $n=1$
parameter,  each $f_i$ is a product of linear forms,
by the Fundamental  Theorem of Algebra.
Hence, $f_1f_2 \cdots f_k$ defines an arrangement in $\mathbb{R} ^1$.
The rows of the table indicate the multiplicity of each 
hyperplane in the~product.

In the context of statistics, the $s_i$ represent nonnegative
integers which summarize an independent and 
identically distributed sample. Namely, $s_i$ is the number of
observations of the $i$th outcome. The sum $s_1+ s_2 + \cdots + s_k$
is the sample size of the experiment. The function
$\,f^s \, = \,f_1^{s_1} f_2^{s_2} \cdots f_k^{s_k}$ is the
{\em likelihood function} of the model with respect to the data.
We think of $s$ as parameters, so we are interested in the situation 
when the model is fixed and the data varies. The vector
$s$ of counts ranges over  $\mathbb{N}  ^k$, or even over $\mathbb{R} ^k$ or $\mathbb{C} ^k$,
and we treat its coordinates $s_i$ as unknowns.
The role of the likelihood function $f^s$ in statistics can be summarized
by referring to the two camps in the history of statistics:
\begin{bulist}
	\item {\bf Frequentists}: Compute $\,{\rm argmax}(f^s)$, i.e.,~solve 
	an optimization problem.
	\item {\bf Bayesians}: \ \ \ \ Compute $\,\int_\gamma f^s dx$, 
	i.e.,~evaluate a certain definite integral.
\end{bulist}
We refer to Sullivant's book \cite[Chapter 5]{als:Sul}
for a discussion of these two perspectives.
The integration problem is reminiscent of the volume computation in the
previous lecture. Let us now discuss the optimization problem. This is the problem 
of maximum likelihood estimation (MLE).
The aim is to maximize $f^s$ over a suitable open subset
of parameters $x$ in $ \mathbb{R} ^n$. To address this problem algebraically,
one studies the map that associates to $s$ the critical points.
This map is an algebraic function. The number of its branches
is known as the {\em maximum likelihood degree} (ML degree);
see \cite{als:HuSt} and \cite[Chapter 7]{als:Sul}. Models of special interest
are those where the ML degree is one. This means that
the MLE is given by a rational function.
Such models were studied recently in \cite{als:DMS}.
The following two examples have this property.

\begin{ex}[Example \ref{als:exflipping} revisited]
	Consider the model where a biased coin is flipped $m$ times.
	The likelihood function $f^s$ has its unique critical point at
	\[  1-\hat x \,\,= \,\,  \frac{1}{m}\cdot \frac{m \cdot s_0+(m-1)\cdot s_1+\cdots +
		1 \cdot s_{m-1} +  0 \cdot s_m}{s_0+s_1+s_2+\cdots +s_m} . \]
	The probability estimates $\,f_i(\hat x)\,$ are alternating products of linear forms
	in the counts $s_0,\ldots,s_m$. The
	linear forms appearing in the numerators are seen in~\eqref{als:eqcoinsproduct}.
\end{ex}

\begin{ex}
	We examine the model that serves as the running example in \cite{als:DMS}.
	It concerns the following simple experiment: {\em Flip a biased coin.
	If it shows head, then flip it one more time.} Here $k=3,\,m=2$ and $n=1$. The 
	model is given algebraically by $\,f_0=x^2$, $f_1=x(1-x)$, and
	$ f_2=1-x$. These three polynomials in $\mathbb{Q} [x]$ sum to $1$. 
	The likelihood function equals
	\[ f^s \,\,=\,\, f_0^{s_0} f_1^{s_1} f_2^{s_2} \,\,=\,\, x^{2s_0+s_1} \cdot (1-x)^{s_1+s_2}.\]
	This is annihilated by the following first-order operator:
	\[ P \,\, = \,\, (x^2-x) \partial \, - \, (2 s_0 + s_1 + s_2) x \,+ s_1 + s_2 \,\in \,D[s_0,s_1,s_2].\]
 By eliminating $x$ and $\partial$ from the $D[s]$-ideal generated by $P$ and $x(1-x)$, we~get
	\[ \mathcal{B}\left( f \right) \,\,=\,\, \biggl\langle \,
	\prod_{k=1}^3 \left( 2s_0 + s_1 + k \right) \cdot \prod_{l=1}^2 \left( s_1+s_2+l \right) \,
	\biggr\rangle. \]
	Thus, the Bernstein--Sato ideal is principal and generated by
	a product of linear forms \eqref{als:typehyperplane}.
	As before, we  recover the linear factors appearing in
	$ \mathcal{B}(f)$: 
	\begin{center}
		\begin{tabular}{r|ccc|l}
			&$f_0$ & $f_1$  & $f_2$ & $\sum$\\ \hline
			$x$ & $2$&$1$&$0$&$3$\\
			$1-x$ & $0$ & $1$&$1$& $2$
		\end{tabular}
	\end{center}
	This table of multiplicities mirrors the formula in \cite[Example 2]{als:DMS} for the
	maximum likelihood estimate in the coin flip model:
\[  \hat x \,\,=\,\, \frac{2s_0+s_1}{2s_0+2s_1+s_2}\,, \qquad
	1 - \hat x  \,\,=\,\, \frac{s_1+s_2}{2s_0+2s_1+s_2} .\]
	Just like in Example \ref{als:exflipping}, the numerators are
	precisely the linear factors in $\mathcal{B}(f)$.
\end{ex}

Our discussion suggests that there is a 
deeper connection between
$D$-module theory and likelihood geometry~\cite{als:HuSt}.
This deserves to be explored. Further, it would
be interesting to study the Bayesian  integrals
$\int_\gamma f^s$ using the $D$-module methods 
from the previous lecture.

\smallskip
We now come to the Holonomic Gradient Method (HGM).
Consider the problem of  maximum likelihood estimation in statistics,
but now for continuous distributions rather than discrete ones.
Our aim is to explain the benefit gained from $D$-module theory.
Indeed, many functions of relevance in statistics are holonomic.
A key idea is to compute and represent the gradient of such a function
from its canonical holonomic representation.

The statistical aim of MLE is to find parameters for which an observed outcome is 
most probable \cite[Chapter 7]{als:Sul}. This can be formulated as an optimization problem, 
namely to maximize the likelihood function. For discrete models, this function
has the form~$f^s$, as seen above. In what follows
we consider the likelihood function for continuous models.

Our goal is to find a local maximum of 
a holonomic function using a variant
of gradient descent.
For the sake of efficiency, our
computations are carried out in the  rational Weyl algebra~$R  = \mathbb{C} (x)\langle \partial \rangle$.
See~\cite[Section 1.4]{als:SST} for the theory  of Gr\"{o}bner bases in $R$. 
Unless otherwise stated, we use the graded reverse lexicographical order $\prec$. 
For $n=2$,~this~gives 
\[ 1 \prec \partial_2 \prec \partial_1 \prec \partial_2^2 \prec 
\partial_1\partial_2 \prec \partial_1^2 \prec \cdots .\]

Let $f(x_1,\ldots,x_n)$ be a real-valued holonomic function and $I$ a $D$-ideal with 
finite holonomic rank such that $I\bullet f=0$. Thus, $R/RI$ is finite-dimensional over $\mathbb{C} (x_1,\ldots,x_n)$. 
In our application, $f$ will be the
likelihood function of a statistical model.
Let $\text{rank}  (I)=m\in \mathbb{N}  _{>0}$. We write
$S=\left\{ s_1,\ldots,s_m \right\}$,
with $s_1 = 1$, for the
set of standard monomials for a Gr\"{o}bner basis of $RI$ in $R$.
By Proposition \ref{als:proppartial},
the $m$ entries of the following vector are holonomic functions
\[F\,\,=\,\,\left(s_1 \bullet f, s_2 \bullet f, \ldots, s_m \bullet f \right)^T. \]
Note that the first entry of $F$ is the given function $f$.
In symbols,  \mbox{$ (F)_1=f$}. Since 
the $D$-ideal $I$ has holonomic rank $m$, there exist unique matrices \linebreak
$P_1,\ldots,P_n \in\mathbb{C} (x_1,\ldots,x_n)^{m\times m}$ such that
\begin{align}\label{als:pfaffsys} \qquad
\partial_i \bullet F \,\,=\,\, P_i \cdot F \qquad \text{for }\,\,i=1,\ldots,n.
\end{align}
The system of linear partial differential equations \eqref{als:pfaffsys} 
is called the {\em Pfaffian system of $f$}. Note that it depends
on the specific $R$-ideal $RI$ and on the chosen term order.
The matrices $P_i$ can be computed as follows.
We apply the division algorithm modulo our Gr\"obner basis
to the operators $\partial_i s_j ,$
for $i \in \{1,\ldots,n\}$ and $j \in \{1,\ldots,m\}$.
The resulting normal form~equals
\[ a^{(i)}_{j1}(x) s_1 \,+\,   a^{(i)}_{j2}(x) s_2 \,+\,  \cdots \,+\,   a^{(i)}_{jm}(x) s_m ,  \]
where the coefficients $a^{(i)}_{jk}$ are rational functions in $x_1,\ldots,x_n$.
This means that the operator $\,\partial_i s_j  - \sum_{k=1}^m a^{(i)}_{jk}(x) s_k\,$ is in
the $R$-ideal~$RI$. From this one sees that the coefficient $a^{(i)}_{jk}(x)$ is the entry 
of the $m {\times} m$ matrix $P_i$ in row $j$ and column~$k$.

We have now reached the following important conclusion. Suppose 
$x$ is replaced by a point $u$ in $\mathbb{Q} ^n$. Here $u$
might be a highly accurate floating point representation of a point in $\mathbb{R} ^n$.
The numerical evaluation of the gradient of~$f$  at $u$ reduces to 
multiplying the vector $F(u) \in \mathbb{Q} ^m$ by matrices $P_i(u)$ with explicit rational entries.
A tacit assumption made here is that $u$ lies in the complement
of the singular locus of the Pfaffian system \eqref{als:pfaffsys} that encodes~$f$.

\begin{ex}[$n=1$] 
	Let $f$ be  a holonomic function annihilated by
		$$ I\,\, =\,\, \langle \,x \partial^3-(x+1)\partial+1 \,\rangle. $$ 
	The generator by itself
	is a Gr\"obner basis for $RI$. The set of standard monomials equals
	$S = \{1,\partial,\partial^2\}$, and this is a $\mathbb{C} (x)$-basis of $R/RI$.
	From $I$ we see that
	\[ \partial^3\bullet f\,\,= \,\,\frac{x+1}{x}\partial \bullet f \,-\, \frac{1}{x}\cdot f.\]
	Let $F=(f,\partial \bullet f,\partial^2 \bullet f)^T$. This yields the
	following Pfaffian system for $f$:
	\begin{align*}
	\partial \bullet F \,\,=\,\,  P \cdot F \quad {\rm where} \quad P \,\,=\,\, 
	\begin{pmatrix}
	0& 1 & 0\\0 & 0 & 1\\ -\frac{1}{x} & \frac{x+1}{x} & 0 
	\end{pmatrix}.
	\end{align*}
	Using notation familiar from calculus,
	for any non-zero real number $u$ we have
	\[ \begin{pmatrix} f'(u) \\ f''(u) \\ f'''(u) \end{pmatrix}
	\,\,= \,\,
	\begin{pmatrix}
	0& 1 & 0\\0 & 0 & 1\\ -\frac{1}{u} & \frac{u+1}{u} & 0 
	\end{pmatrix} \cdot \begin{pmatrix} f(u) \\ f'(u) \\ f''(u) \end{pmatrix}. \]
	This matrix-vector formula is useful for the design of numerical algorithms.
\end{ex}

Given a holonomic function $f$, represented by 
a holonomic $D$-ideal, we are interested in the following two questions.
The first of these was already discussed in the previous lecture.
\begin{numlist}
	\item  How to evaluate $f$ at a point with the help of the knowledge of $I$?
	\item  How to find local minima of $f$ with the help of the knowledge of $I$?
\end{numlist}
We first describe how to evaluate the holonomic function $f$ at a 
point $\tilde{x}$ by a first order approximation. Assume we are able 
to numerically evaluate $f$ at some particular point $x^{(0)}$, 
depending on the precise situation. 
Choose a path $x^{(0)} \to x^{(1)} \to \cdots \to x^{(K)}=\tilde{x}$, 
with $x^{(k+1)}$ sufficiently close to~$x^{(k)}$ for all $k=0,\ldots,K-1$
and such that the path does not cross the singular locus of the Pfaffian 
system of~$f$. The following algorithm is referred to as the

\bigskip

\noindent\textbf{Holonomic Gradient Method (HGM).}
\begin{numlist}
	\item Compute a Gr\"{o}bner basis of $RI$ in 
	the rational Weyl algebra $R$.  \vspace{-0.1cm}
	\item Compute the
	set of standard monomials $S$ and
	the Pfaffian system~\eqref{als:pfaffsys}. \vspace{-0.1cm}
	\item Evaluate $F$ at one point $x^{(0)}$ and 
	denote the result by $\bar{F}$. Set $k=0$. \vspace{-0.1cm}
	\item Approximate the value of the vector $F$ at $x^{(k+1)}$ by
	its first-order Taylor polynomial, and denote the result again by $\bar{F}$:
	\begin{align*} 
	F\left(x^{(k+1)}\right) &\,\,\approx\,\, F\left(x^{(k)}\right) \,+\, \sum_{i=1}^n \left(x_i^{(k+1)}-x_i^{(k)}\right) \cdot \left(\partial_i \bullet F\right)\left(x^{(k)}\right)\\ 
	&\,\, = \,\, F\left(x^{(k)}\right) \,+\, \sum_{i=1}^n \left( x_i^{(k+1)} -x_i^{(k)} \right) \cdot P_i\left( x^{(k)}\right) \cdot \bar{F}. 
	\vspace{-0.2cm}
	\end{align*}  
	\item Increase the value of $k$ by $1$. If $k<K$, return to step 4. Otherwise stop.
\end{numlist}
Steps 1 to 3 need to be carried out only once for $(f,I)$.
The output of this algorithm is a vector $\bar F$ that approximates
$F(\tilde x)$. The first coordinate of $F(\tilde x)$ is the
desired scalar $f(\tilde x)$. Hence
the first coordinate of $\bar{F}$ is our
approximation.

\begin{remark}
	To turn the HGM into a practical algorithm, it is essential
	to incorporate some knowledge from numerical analysis.
	For instance, there is a lot of freedom
	in choosing the numerical approximation method in step 4. 
	Nakayama et al.~\cite{als:HGM} use the Runge--Kutta method of fourth order. 
	Another possibility is to use a second order Taylor approximation. 
	Here one computes the Hessian of~$f$ also by means of the Pfaffian system of~$f$.
\end{remark}

\begin{remark}
	In practical applications, the Gr\"obner basis computation in step~1 
	may not provide results within a reasonable time, but
	sometimes a partial Gr\"obner basis for $RI$ suffices to certify holonomicity. 
	From this one gets a finite superset $S$  of the unknown 
	set of true standard monomials.
	The set $S$ spans $R/RI$, but it may not be linearly independent.
	In that case,  one can still compute a Pfaffian system, but
	the $P_i$ are not necessarily unique anymore. This relaxation
	might work well in practice.
\end{remark}

We are now endowed with all necessary tools for finding a
local minimum of the holonomic function $f$. 
As before, $f$ is encoded by an annihilating $D$-ideal $I$ with 
finite holonomic rank. This encoding is the input to the next algorithm.

\bigskip

\noindent\textbf{Holonomic Gradient Descent (HGD).}
\begin{numlist}
	\item Compute a Gr\"{o}bner basis of $RI$ in 
	the rational Weyl algebra $R$.  \vspace{-0.1cm}
	\item Compute the set of standard monomials $S$ and
	 the Pfaffian system \eqref{als:pfaffsys}.  \vspace{-0.1cm}
	\item Numerically evaluate $F(x^{(0)})$ at some starting point 
	$x^{(0)}$ and put $k=0$. Denote this value by $\bar{F}$. 
	The evaluation method is chosen to be adapted to the problem.  \vspace{-0.1cm}
	\item For $i=1,\ldots,n$, evaluate  the first coordinate of
	$ \,P_i(x^{(k)})\bar{F}$. Let $\bar{G}$ be~the vector
	of these $n$ numbers.
	This  approximates the gradient $\nabla f$ at $x^{(k)}$~since
	\[ \,\partial_i \bullet f \,\,= \,\,\left( \partial_i\bullet F\right)_1
	\,\,=\,\, \left(P_i \cdot F\right)_1.\]
	\item If a termination condition of the iteration is satisfied, stop. 
	Otherwise go to step~6. \vspace{-0.1cm}
	\item Put $x^{(k+1)}=x^{(k)}-h_k \bar{G}$, where $h_k$ is 
	an appropriately chosen step length. \vspace{-0.5cm}
	\item Numerically evaluate $F$ at $x^{(k+1)}$ by step 4 of the HGM
	and set this value to $\bar{F}$. 
	Increase the value of the index $k$ by one and return to step 4 above.
\end{numlist}

The algorithm returns a point $x^{(k)}$ along with 
the value of $F$ at that point. The first entry of this output is a
numerical approximation of a local minimum of the
holonomic function~$f$. Again, one should be aware that, in general, this 
algorithm works only within connected components contained 
in the complement of the singular locus of the Pfaffian system of $f$.

In order to develop a practical implementation,
and to assess the quality of the method, one needs
some expertise from numerical analysis.
The choices one makes can make a huge difference.
For instance, consider the choice of the step size $h_k$.
This is a well-studied subject in numerical optimization,
and there are various standard recipes for carrying
out gradient descent. In current applications to data science,
stochastic versions of gradient descent play a major role, and it would
be very nice to connect $D$-modules to these developments.

The applicability of HGD arises from the fact that many
distributions that are used in practice are given by holonomic functions. 
One example is the cumulative distribution function of the largest 
eigenvalue of a Wishart matrix, cf.~\cite{als:HNTT}. 
Another relevant holonomic function 
is the likelihood function of sampling matrices in $\text{SO}(3)$. In
what follows we present in detail an example that stood
at the beginning of the development of HGM and HGD.

\begin{ex}[The Fisher--Bingham distribution \cite{als:HGM}]
	Let \[ \,\mathbb{S}^n(r) \,\,=\,\, \left\{ x\in \mathbb{R} ^{n+1} \mid \lVert x \rVert =r \right\}\,\]
	denote the $n$-sphere of radius $r$.
	Let  $x\in \mathbb{R} ^{(n+1)\times (n+1)}$ be a symmetric matrix and $y\in \mathbb{R}^{n+1}$ a row vector. 
	Let $\lvert dt \rvert$ denote the standard measure on $\mathbb{S}^n(r)$.
	The {\em Fisher--Bingham integral} is
	\[ F(x,y,r) \,\,= \,\, \int_{\mathbb{S}^n(r)}\!\! \exp\left( t^Txt+yt \right) \vert dt\vert.\]
	This is a function in $\binom{n+1}{2} + (n+1) + 1$ unknowns, since $x_{kl}=x_{lk}$.
	It is shown in \cite[Theorem~1]{als:HGM}  that the function
	$F(x,y,r)$ is holonomic.
	More precisely, the following operators
	annihilate the Fisher--Bingham integral and generate a $D$-ideal of finite holonomic rank:
    \begin{align*}
	\sum_{i=1}^{n+1} \partial_{x_{ii}} - r^2\, ,\quad
	r\partial_r-2\sum_{i\leq j}x_{ij}\partial_{x_{ij}}-\sum_i y_i \partial_{y_i}-n \, , \quad
	\partial_{x_{ij}}-\partial_{y_i}\partial_{y_j} \,\,\,{\rm for} \,\,\, i \leq j,
	\smallskip \\
	x_{ij}\partial_{x_{ii}} + 2(x_{ji}-x_{ii})\partial_{x_{ij}}-x_{ij}\partial_{x_{jj}}
	+\sum_{k\neq i,j} (x_{jk}\partial_{x_{ik}}-x_{ik}\partial_{x_{jk}}) + y_j \partial_{y_i} - y_i\partial_{y_j},
	\end{align*}
	where $i<j$ in the second line.
	For a proof, see \cite[Theorems 2,3]{als:HGM}. For $n=1,2$, these 
	operators generate a holonomic $D$-ideal, see~\cite[Proposition~1]{als:HGM}. 
	We now define the {\em Fisher--Bingham distribution} on the unit sphere $\mathbb{S}^n(1)$.
	This depends on the parameters $x,y$ and  has 
	probability density function 
	\[ p(t \! \mid \! x,y) \,\,\,= \,\,\, F(x,y,1)^{-1} \cdot \exp (t^T xt+yt). \]
	In other words, the Fisher--Bingham distribution plays the role of the
	Gaussian distribution on the sphere, and the Fisher--Bingham integral
	$F(x,y,1)$ is its normalizing constant.
	
	We now explain the inference problem to be solved.
	Let $\{t(1),\ldots,t(N)\}$ be an independent 
	and identically distributed sample of size~$N$ drawn from the unit sphere $\mathbb{S}^n(1)$.
	The statistical aim is to estimate the parameters
	$x=(x_{ij})$ and $y=(y_i)$ from the given sample. 
	The standard method to do so is MLE.
	We seek to maximize the {\em likelihood function} 
	\[ (x,y) \,\,\mapsto \,\, \prod_{\nu =1}^{N} p(t(\nu) \! \mid  \! x,y) .\]
	For the Fisher--Bingham model, this 	is equivalent to minimizing the function
	\begin{align}\label{als:FBexp}
	F(x,y,1) \cdot \exp \left(  -\!\!\sum_{1\leq i\leq j \leq {n+1}} S_{ij}x_{ij} 
	\,\,- \sum_{1\leq i \leq {n+1}} S_iy_i \right).
	\end{align}
	Here the quantities $S_i$ and $S_{ij}$ are real constants 
	that are easily computed  from the sample points $t(i)$.
	Namely, they are the coordinates of the sample mean and the sample
	covariance matrix:
	\[ S_i \,= \,N^{-1}\sum_{\nu=1}^Nt_i(\nu) \quad {\rm and} \quad
	S_{ij}\, =\, N^{-1} \sum_{\nu=1}^{N}t_i(\nu)t_j(\nu). \]
	The function  \eqref{als:FBexp} is a product of two holonomic functions.
	Hence, it is a holonomic function in the unknowns $x=(x_{ij})$ and $y=(y_i)$. 
	Furthermore, since our model is an exponential family,
	the logarithm of~\eqref{als:FBexp} is a convex function.
	This means that a local minimum is already a global one.
	Our task is therefore to find a local minimum of~\eqref{als:FBexp} using HGD.
	
	The authors of \cite{als:HGM} present two specific data sets
	and they demonstrate the use of HGD for this input.
	The data and some code in the computer algebra system {\tt Risa/Asir}  
	are provided at the website 
$$ \hbox{	
	{\tt http://www.math.kobe-u.ac.jp/OpenXM/Math/Fisher-Bingham/}.} $$
	
	One of the data sets is the following ``astronomical data''. 
	Here $n= 2$ and
		\begin{align*}
	& S_1=-0.0063, & S_{11}&=0.3199, & S_{22}&=0.3605,\\
	& S_2=-0.0054, & S_{12}&=0.0292, & S_{23}&=0.0462,\\
	& S_3=-0.0762,  & S_{13}&=0.0707, &S_{33}&=0.3276.
	\end{align*}
	The starting point is found by minimizing with a quadratic
	approximation of $F(x,y,1)$, with step size set at $0.05$.
	Running the HGD revealed the minimum objective function
	value $\approx 11.6857$. The maximum likelihood parameters are
	\begin{align*}
	x\,=\,\begin{pmatrix}
	-0.161 & 0.3377/2 & 1.1104/2\\
	0.3377/2 & 0.2538 & 0.6424/2\\
	1.1104/2 & 0.6424/2 & -0.0928
	\end{pmatrix}, \quad y\,=\,(-0.019,-0.0162,-0.2286).
	\end{align*}
	We reproduced this result, but this did take some effort.
\end{ex}

\medskip

In conclusion, we have argued that holonomic functions
arise in many contexts, notably 
in geometry and statistics. The manipulation of these functions can
be done by algorithms from the theory of $D$-modules.
Implementations already exist, and they are available
in a wide range of computer algebra systems. While the further development of
the symbolic computation tools is important, a significant new
opportunity lies in advancing the connection
to numerical algebraic geometry. Efficient numerical methods
for $D$-modules and holonomic functions have a clear
potential for future impact in scientific computing and data science.
These lectures offered a very first glimpse at the underlying mathematics.


\bigskip

\section*{Problems}

In this section, we offer some ideas for hands-on activities. These were discussed
in an afternoon session during the Berlin school.
The items range from easy exercises to 
challenging questions that suggest research projects.
We leave it to our readers to decide which is which.
Some hints and solutions are found below. 

\medskip

\begin{numlist}
	\item Let $M$ be a $D$-module which is finite-dimensional as a $\mathbb{C} $-vector space.
	Show that $M=0$. \hfill Hint: Can the commutator of two $n \times n$ matrices
	be the identity matrix?
	\item Find bases of solutions for the following three second-order linear differential equations:
	\[ x f'' + f' \,=\, 0 \quad {\rm and} \quad
	g'' \,=\, 4 x^2 g \quad {\rm and} \quad 
	x^2 h'' - 3 x h' + 4h \,=\, 0. \]
	\item For each of the following three functions  $u,v,w$ in one variable $x$,
	find a linear ordinary differential equation with polynomial coefficients that is 
	satisfied by that function:
	\[ u(x)\, =\, x^{3/5} \cdot {\rm log}(x)^2
	,\quad
	v(x) \,=\, {\rm sin}(x)^5
	,\quad
	w(x) \,=\, (1 + x^4)\cdot {\rm exp}(x). \]
	\item A $\mathbb{C} $-basis of the Weyl algebra 
	$D = \mathbb{C}  \left[ x_1,\ldots,x_n\right] \langle \partial_1,\ldots,\partial_n\rangle$
	consists of the normal monomials $x^a \partial^b$, where $a,b \in \mathbb{N}  ^n$.
	Find the formula for expressing
	the operator $\partial^b x^a$ in that~basis.	
	\item Let $n=3$. Compute the distraction $\,\widetilde I\,$ in $\mathbb{C} [\theta_1,\theta_2,\theta_3]$ 
	of	 the $D$-ideal
	\[I \,\,=\,\, \langle \partial_1^4, \partial_2^4, \partial_3^4,
	\partial_1 \partial_2^2 \partial_3^3,
	\partial_1^2 \partial_2^3 \partial_3, 
	\partial_1^3 \partial_2 \partial_3^2 \rangle. \]
	Also, find a 
	$\mathbb{C} $-basis   for the vector space  $\,{\rm Sol}(I) = {\rm Sol}(\widetilde I)$.
	\item Find a canonical holonomic representation for the bivariate function
	\[ f(x,y) \,\, = \,\,  e^{x \cdot y} \cdot {\rm sin} \frac{y}{1 + y^2}. \]
	\item Using the integration ideal as in Definition \ref{als:defintegrationideal}, find
	an operator in $D_1$ that annihilates the following function in one variable:
	\[ \qquad F(x) \,\, = \,\, \int_{0}^{+ \infty}\!\! f(x,y) dy,
	\qquad \text{where $f$ is the function in Problem 6.} \]
	\item Construct a rational function $r$ by taking the ratio of your two favorite
	polynomials in two variables. Compute $I = \text{Ann} _D(r)$
	and determine the singular locus ${\rm Sing}(I) \subset \mathbb{C} ^2$.
	\item Let $f$ be a holonomic function in one variable.
	Prove that its reciprocal\linebreak$1/f$ is holonomic if and only if
	the logarithmic derivative $f'/f$ is an \mbox{algebraic} function.
	\item {\em For those who like sheaves}:
	Why is a vector bundle together with a flat connection a module over the sheaf $\mathcal{D}$?	
	Actually, what are these objects over the projective line~$\mathbb{P}^1$?
	\item {\em For those who like toric geometry:}
	Pick your favorite projective toric manifold
	and write the presentation ideal
	(Stanley--Reisner plus linear forms) of its Chow ring
	in $\mathbb{C} [\partial_1,\ldots,\partial_n]$ and  also in
	$\mathbb{C} [\theta_1,\ldots,\theta_n]$.
	Determine the solution spaces in both cases.
	\item {\em For those who like the Hodge theory of matroids:}
	Pick your favorite matroid
	and write the presentation ideal
	(Stanley--Reisner plus linear forms) of its Chow ring
	in $\mathbb{C} [\partial_1,\ldots,\partial_n]$ and also in
	$\mathbb{C} [\theta_1,\ldots,\theta_n]$.
	Determine the solution spaces in both cases.
	\item Stafford's Theorem states that every $D$-ideal can be generated by
	two elements.
	Let $n=4$ and identify two differential operators that generate the $D$-ideal
	$\langle \partial_1,\partial_2,\partial_3,\partial_4\rangle$. Then, do the same for the
	$D$-ideal $I$ in Example~\ref{als:eqGauss}, for some choices of $a,b,c \in \mathbb{Z} $.
	\item Consider the general algebraic equation of degree five in one variable:
	\[  x_5 t^5 + x_4 t^4 + x_3 t^3 + x_2 t^2 + x_1 t + x_0 \,\,= \,\, 0 .\]
	Write the roots $t_1,\ldots,t_5$ as a holonomic function of the coefficients $x_i$.
	Restrict your holonomic system to a two-dimensional linear subspace in 
	the $\mathbb{C} ^6$ of coefficients.
	\item Let $f,g,h$ be as in  Problem 2.
	According to Proposition \ref{als:prodhol}, the functions
	$fg,\,fh,\,gh,\,f{+}g,\,f{+}h,\,g{+}h$ are holonomic.
	Find operators in $D_1 $ that annihilate these functions.
	\item Compute a Gr\"obner basis in $R$ for the annihilator of the
	function $f(x,y)$ in Problem~6. Determine the Pfaffian system (\ref{als:pfaffsys}).
	Verify that $P_1$ and $P_2$ satisfy \cite[Equation (1.35)]{als:SST}.
	\item Write the likelihood function $f^s$
	for the {\em random censoring model} in \cite[Example 7.1.5]{als:Sul}.
	Compute the $s$-parametric annihilator
	${\rm Ann}_{D[s]}(f^s)$ and the Bernstein--Sato ideal~$\mathcal{B}(f)$.
	\item Let $n=4$ and let $I$ be the left ideal in $D_4$ generated by the four operators
	\begin{align*}
	3 x_1 \partial_1 + 2 x_2 \partial_2 + x_3 \partial_3 -3,&\quad 	
	x_1 \partial_2 + 2 x_2 \partial_3 + 3 x_3 \partial_4,\\
	(3 x_2 \partial_1 + 2 x_3 \partial_2 + x_4 \partial_3)^4, &\quad
	 x_2 \partial_2 + 2 x_3 \partial_3 + 3 x_4 \partial_4 .
	\end{align*}
	Show that $I$ is holonomic and determine its rank.
	Compute the characteristic variety and the singular locus.
	Explain their irreducible components.
	\item Let $n=9$, where the Weyl algebra generators
	 $x_{ij}$ and $ \partial_{ij}$ are entries of  $3 \times 3$ matrices
	respectively.
	Let $P$ be the prime ideal in $\mathbb{C} [x_{ij}]$ that defines the 
	group ${\rm SO}(3)$ in $\mathbb{C} ^{3 \times 3}$.
	Let $I$ be the $D$-ideal generated by $P$ and 
	\[ \biggl\{ \,\sum_{k=1}^3\, (x_{ki} \partial_{kj} - x_{kj} \partial_{ki})  
	\,\, : \,\, 1 \leq i < j \leq 3 \,\biggr\}. \]
	Show that $I$ is holonomic and Weyl-closed. 
	Compute its rank and characteristic variety.
\end{numlist}

\section*{Solutions and Hints}

\begin{numlist}
	\item Let $M\in \text{Mod}(D)$ with $\dim _{\mathbb{C} } M =m$. Since $[\partial_i,x_i]=\text{id}_M$,
	the trace of the commutator is given by $\text{tr}\left([\partial_i,x_i]\right) = m \cdot 1$. 
	On the other hand, $\partial_i,x_i \in \text{End}_{\mathbb{C} }(M)$ are described by matrices 
	$P_i,X_i \in \mathbb{C} ^{m\times m}$. Hence, 
	$\text{tr} \left([\partial_i,x_i] \right) = \text{tr}\left(P_iX_i-X_iP_i \right) = \text{tr} (P_iX_i)-\text{tr}\left(X_iP_i\right)=0$
	and therefore~$m=0$. 	
	\item 
		We start with $P_f =x  \partial^2 + \partial$ and  $I_f=\langle P_f \rangle$. 
		A computation shows that $\text{in} _{(-w,w)}(I_f)=I_f$ for all $w\in \mathbb{R} ^n$. 
		Therefore, $I_f$ is torus-fixed and $\text{Sol}(I_f)=\text{Sol} \left( \widetilde{I_f}\right)$, where 
$\widetilde{I_f} = \langle \theta^2\rangle$. Therefore, $\text{Sol}(I_f)=\mathbb{C}  \left\{ 1,\log(x) \right\}$.

The second ODE corresponds to the operator		
		 $P_g = \partial^2 - 4x^2$.  The lowest order terms of the solutions to $P_g$ can be computed
		by Proposition~\ref{als:propdistrac}.
		
		Denote by $P_h= x^2\partial^2 - 3x\partial + 4$. We observe  that the ideal generated by 
		\mbox{$P_h = \theta^2 - 4 \theta +4 = (\theta-2)^2$} is a Frobenius ideal. 
		The inverse system at $2$ is $\,Q_2^{\perp} =\mathbb{C}  \{1,x\}$. 
		By Theorem~\ref{als:solutionFrob}, we get
		$\text{Sol} \left( \langle P_h\rangle \right) =\mathbb{C} \left\{ x^2,x^2\log(x) \right\}$.
	
	One can reproduce these results by the following  {\tt Mathematica} code:
	\begin{verbatim}
	DSolve[x*y''[x] + y'[x] == 0, y[x], x] 
	DSolve[y''[x] - 4*x^2*y[x] == 0, y[x], x] 
	DSolve[x^2*y''[x] - 3*x*y'[x] + 4*y[x] == 0, y[x], x] 
	\end{verbatim}
	\item The reader is encouraged to compute the derivatives and 
	then identify a relation between them. 
	Alternatively, the following code in {\tt Mathematica} computes annihilating differential operators for $u,v,w$:
	\begin{verbatim}
	<< RISC`HolonomicFunctions`
	u = x^(3/5)*(Log[x])^2
	annu = Annihilator[u, Der[x]]
	v = (Sin[x])^5
	annv = Annihilator[v, Der[x]]
	w = (1 + x^4)*Exp[x]
	annw = Annihilator[w, Der[x]]
	\end{verbatim}
	\item
	Start with $n=1$ and then extend.
	For any nonnegative integers $a$ and $b$, 
	\[ \partial^b x^a\,\,=\,\, \sum_{i\geq 0} \frac{a!b!}{i!(a-i)!(b-i)!} x^{a-i}\partial^{b-i} ,\]  
	where negative powers are $0$ and zero powers are $1$. 
	\item Similar to Example~\ref{als:staircase}, consider the staircase under the 
	monomial ideal. This is the cover picture of a text book by Ezra Miller
	and Bernd~Sturmfels.
	\item 	The following code in {\tt Mathematica}
	computes an annihilating ideal for $f$:
	\begin{verbatim}
	<< RISC`HolonomicFunctions`
	f = Exp[x*y]*Sin[y*1/(1+y^2)]
	ann = Annihilator[f,{Der[x],Der[y]}]
	\end{verbatim}
	We find that the operators 
	$\partial_{x}-y$ and $(y^{10}{+}3 y^8{+}2 y^6{-}2 y^4{-}3 y^2{-}1) \partial_y^2 $ $+({-}2 x 
	y^{10}{-}6 x y^8{-}4 x y^6{+}4 x y^4{+}6 x y^2{+}2 x{+}2 y^9{-}12 y^5{-}16 y^3{-}6 y) 
	\partial_y{+}(x^2 y^{10} $ $ +3 x^2 y^8{+}2 x^2 y^6{-}2 x^2 y^4{-}3 x^2 y^2{-}x^2{-}2 x 
	y^9{+}12 x y^5{+}16 x y^3{+}6 x y{+}y^6{-}3 y^4 $ \linebreak $ +3 y^2{-}1)$ 
	annihilate $f$.
	It remains to specify sufficiently many initial conditions and to prove that 
	these two operators generate a holonomic $D$-ideal.
	\item The theoretical argument is similar to that in the proof of Proposition~\ref{als:inthol}. 
	For the computation we can use the commands {\tt CreativeTelescoping} and {\tt ApplyOreOperator} 
	of the {\tt HolonomicFunctions} package in {\tt Mathematica}.
	\item Let $r=p/q$ with $p=x^2+xy, \,q=y^2\in \mathbb{C} [x,y]$. 
	Running the following code in {\tt Singular}, using the libraries \cite{als:AL,als:A,als:LM}, 
	solves the problem:
	\begin{verbatim}
	LIB "dmod.lib";	LIB "dmodapp.lib";	LIB "dmodloc.lib";
	ring r=0,(x,y),dp; setring r;
	poly p=x^2+x*y;	poly q=y^2;
	def an=annRat(p,q); setring an; LD;
	isHolonomic(LD);
	DsingularLocus(LD);
	def CV=charVariety(LD); setring CV;
	charVar;
	\end{verbatim}
	\item For a proof, we refer to the article~\cite{als:HS} of Harris and Sibuya.
	\item Vector bundles with flat connection over $\mathbb{P}^1$  correspond to those over~$\mathbb{P}^{\text{an}}$. 
	By the Riemann--Hilbert correspondence, such a vector bundle is determined by its monodromy data. 
	Since the Riemann sphere is simply connected, it follows that vector bundles with flat connection 
	on $\mathbb{P}^1$ are classified by globally free sheaves of finite rank with the natural action of~$\mathcal{D}$.
	\item We consider the Hirzebruch surface whose Chow ring
	has the presentation ideal $\,I \,=\,\langle \partial_1 \partial_3, \partial_2 \partial_4, 
	\partial_1+ 2 \partial_2 - \partial_3 , \partial_2 - \partial_4 \rangle$.
	Then the solution space is
	\[ {\rm Sol}(I) \,\,= \,\,
	\mathbb{C} \bigl\{ -x_1^2+x_1 x_2+x_1 x_4+x_2 x_3+x_3^2+x_3 x_4,\,
	x_1 + x_3, \,x_2 + 2 x_3 + x_4 ,\,1\,\bigr\}. \]
	The quadratic form is the {\em volume polynomial} \cite[Remark~3.6.14]{als:SST} which 
	generates ${\rm Sol}(I) \simeq \mathbb{C}^4$ as a cyclic module over 
	$\mathbb{C}[\partial_1,\partial_2,\partial_3,\partial_4]$. If we replace each $\partial_i$
	by $\theta_i = x_i \partial_i$, then the solution space is the same but with
	each variable $x_i$ replaced by the corresponding logarithm ${\rm log}(x_i)$.
	\item We refer to \cite[Definition 1.1]{als:Eur} for
	Chow rings of matroids. Again, the ideal is generated by
	squarefree monomials and linear forms. Eur~\cite{als:Eur}
	gives an explicit formula for the volume polynomial. Just like in the
	toric case, this polynomial
	generates the solution space as a $\mathbb{C}[\partial]$-module.
	\item 
	The following code in {\tt Macaulay2} computes two generators, 
	which are guaranteed to generate the given ideal $I$ over the rational Weyl algebra. 
	\begin{verbatim}
	loadPackage "Dmodules"
	D=QQ[x1,x2,x3,x4,d1,d2,d3,d4,
	WeylAlgebra=>{x1=>d1,x2=>d2,x3=>d3,x4=>d4}];
	D
	I=ideal(d1,d2,d3,d4)
	stafford I
	\end{verbatim}
	By computing reduced Gr\"{o}bner bases over $D_4$, one then can check that 
	the  two obtained operators indeed generate $I$
	as a left module over $D_4$.
	We invite our readers to tackle the challenge of computing two generators for the
	$D$-ideal $I$ in Example~\ref{als:eqGauss}, for some choices of $a,b,c \in \mathbb{Z} $.
	\item Let $f(t,x_0,\ldots,x_5)=  x_5 t^5 + x_4 t^4 + x_3 t^3 + x_2 t^2 + x_1 t + x_0 $.
	The five roots are algebraic functions $t_1,\ldots,t_5$ in the six variables $x_0,\ldots,x_5$. 
	Let $t$ be one out of these five functions. We have
		 $t(\lambda x_0,\ldots, \lambda x_5) = t(x_0,\ldots, x_5)$
		 since $f$ and $\lambda f$ have the same roots.
	By taking the derivative with respect to~$\lambda$, we see that $\theta_0+\cdots+\theta_5$ 
	annihilates~$t$, where $\theta_k=x_k\partial_k$. We  similarly observe 
	$t(x_0,\lambda, x_1,\ldots, \lambda^5 x_5)= \lambda \cdot t(x_0,\ldots,x_5)$, which implies that
	$1-\theta_1-2\theta_2-\cdots - 5\theta_5$ annihilates~$t$. 
	Now, consider $t$ as a simple root of the holomorphic function~$f$. 
	By the Residue Theorem from complex analysis,
	\[ t\,\,=\,\,\int_{\gamma} \frac{z f'}{f}dz\] 
	for a suitable integration cycle $\gamma$ in the complex plane. 
	From this integral representation, with the help of $A$-hypergeometric series, 
	it follows that 
	\[ \left\{ \partial_i \partial_j - \partial_k \partial_l \mid i+j = k+l \right\} \,\subseteq \,\text{Ann} _{D_6}(t).\]
	We refer the reader to~\cite{als:SAhyp} for details. See also the opening section in \cite{als:SST}.
	We note that the package \mbox{{\tt HolonomicFunctions}} in {\tt Mathematica} is able compute 
	an annihilator of an algebraic function, which is implicitly given.
	The restriction of this ideal to a two-dimensional linear subspace in 
	the six-dimensional space of coefficients can be run in {\tt Singular} 
	using the command {\tt restrictionIdeal} in the library {\tt dmoddapp.lib} \cite{als:AL}.
	\item For a recipe, see the proof of Proposition~\ref{als:prodhol}. 
	Alternatively, the commands {\tt DFinitePlus} and {\tt DFiniteTimes} in the 
	{\tt Mathematica} package {\tt Holonomic Functions} compute the annihilators 
	of the sum and product of functions, taking only the annihilators 
	of the single functions as an input.
	\item For computing the Pfaffian system, we recommend using the commands 
	{\tt OreGroebnerBasis} and {\tt OreReduce} for Gr\"{o}bner basis computations in the rational Weyl algebra. 
	They belong to the package {\tt HolonomicFunctions}~in {\tt Mathematica}.
	Do verify that $P_1$ and $P_2$ satisfy \cite[Equation (1.35)]{als:SST}.
	\item Fix $n=3$ and $k=4$. The random censoring model is parametrized by 
\[	\begin{matrix}
	f_1 \,=\, \frac{x_3}{x_1+x_2+x_3} , & \quad &
f_2 \,=\, \frac{x_1x_3}{(x_2+x_3) (x_1+x_2+x_3)}, \smallskip \\
f_3  = \frac{x_2 x_3}{(x_1+x_3)(x_1+x_2+x_3)} ,& \quad & 
	f_4  = \frac{x_1 x_2 (x_1+x_2+2 x_3)}{(x_1+x_3) (x_2+x_3) (x_1+x_2+x_3)} .
	\end{matrix} 
\]	
	Geometrically, this statistical model is the cubic surface
	$\{2 f_1 f_2 f_3+f_2^2 f_3+f_2 f_3^2-f_1^2 f_4+f_2 f_3 f_4 = 0 \}$
	inside the tetrahedron $\{f_1+f_2+f_3+f_4 = 1\}$. Its likelihood function is
	$f^s = f_1^{s_1} f_2^{s_2} f_3^{s_3} f_4^{s_4}$.
To compute $\,\text{Ann}_{D[s]}\left( f^s \right)$,
we must adapt \cite[Algorithm 5.3.15]{als:SST} to the case of rational functions.
	\item We provide the following code in {\tt Singular}, in order to solve this exercise.
	\begin{verbatim}
	LIB "dmod.lib"; LIB "dmodapp.lib"; 
	LIB "dmodloc.lib"; LIB "primdec.lib";
	int n=4; def D=makeWeyl(n); setring D;
	ideal I=3*x(1)*D(1)+2*x(2)*D(2)+x(3)*D(3)-3, 
	(3*x(2)*D(1)+2*x(3)*D(2)+x(4)*D(3))^4,
	x(1)*D(2)+2*x(2)*D(3)+3*x(3)*D(4), 
	x(2)*D(2)+2*x(3)*D(3)+3*x(4)*D(4);
	isHolonomic(I);
	DsingularLocus(I);
	def CV = charVariety(I); setring CV;
	charVar;
	list pr=minAssGTZ(charVar);	pr;
	\end{verbatim}
	\item For the $\text{SO}(2)$ case, computations can be easily run using 
	a computer algebra software.
	For the $\text{SO}(3)$ case, computations get highly intensive. 
	We refer the reader to an article of Koyama~\cite{als:K}. Let us draw the reader's attention
	to the following lemma of this article. 
	{\em Let $I$ be a holonomic $D$-ideal. If $\,\text{in}_{(0,e)}(I)\,$ is prime, then $I$ is maximal.}
	This statement may be used in order to investigate if a holonomic ideal contained 
	in the annihilator of a function already presents the full annihilator. 
	The rank computation for the $\text{SO}(3)$ case is carried out in the article 
	\cite{als:ALSS} by investigating the holonomic dual of that $D$-ideal. 
	Moreover, a generalization to compact Lie groups other than 
	$\text{SO}(n)$ can be found therein.
\end{numlist}

\bigskip \bigskip

\textbf{Acknowledgments.} A number of people
helped us with the material presented here. We are grateful to
Michael F. Adamer, Paul G\"{o}rlach, Alexander Heaton, Roser Homs Pons, 
Christoph Koutschan, Christian Lehn, Viktor Levandovskyy, 
Andr\'{a}s C. L{\H o}rincz, Marc Mezzarobba, and Emre C. Sert\"{o}z.

\bigskip \medskip

\textbf{After-effects.} We are happy to report that our lecture notes, made available
in first version on the {\tt arXiv}
in October 2019,
had some productive consequences already. One of these is our article \cite{als:ALSS}.
Andreas Bernig applied the theory of holonomic functions for completing his proof of a conjecture by Joe Fu, 
establishing a link between integral geometry and combinatorics.
Together with Robin van der Veer, we are currently working on a better structural 
understanding of the connection between Bernstein--Sato theory and MLE in statistics.

\bibliographystyle{alpha}
\bibliography{biblio}


\end{document}